\documentclass[12pt]{article}
\usepackage{layout}
\usepackage{url}
\usepackage{proof}
\usepackage{mathrsfs}
\usepackage{amsmath,amsthm,amssymb}
\usepackage{cleveref}
\usepackage[dvipdfmx]{graphicx}
\usepackage[all]{xy}
\usepackage{color}
\usepackage{indentfirst}

\numberwithin{equation}{section}

\newtheorem{thm}{Theorem}[section]
\newtheorem{prop}[thm]{Proposition}
\newtheorem{lemm}[thm]{Lemma}
\newtheorem{cor}[thm]{Corollary}
\newtheorem{defi}[thm]{Definition}

\newtheorem{rmk}[thm]{Remark}

\def\fk#1{\mathfrak{#1}}
\def\cal#1{\mathcal{#1}}
\def\scr#1{\mathscr{#1}}
\def\rm#1{\mathrm{#1}}
\def\bb#1{\mathbb{#1}}
\def\sf#1{\mathsf{#1}}

\def\wt#1{\widetilde{#1}}
\def\ol#1{\overline{#1}}
\def\wong#1{\colorbox[gray]{0.7}{\textcolor{white}{\text{#1}}}}
\def\wonlg#1{\colorbox[gray]{0.9}{\textcolor{white}{\text{#1}}}}
\def\italicrd#1{\makebox[4.5em][l]{\vphantom{#1}\rotatebox{70}{\scalebox{0.5}[1.5]{\rotatebox{-42}{\smash{\rlap{\text{#1}}}}}}}}
\def\italicld#1{\makebox[4.5em][l]{\vphantom{#1}\rotatebox{35}{\scalebox{1.9}[0.5]{\rotatebox{-70}{\smash{\rlap{\text{#1}}}}}}}}
\def\italicside#1{\makebox[4.5em][l]{\vphantom{#1}\rotatebox{30}{\scalebox{0.7}[1.3]{\rotatebox{-42}{\smash{\rlap{\text{#1}}}}}}}}

\newcommand{\ds}{\displaystyle}
\newcommand{\QQ}{\mathbb{Q}}
\newcommand{\Qlb}{\overline{\mathbb{Q}}_\ell}

\newcommand{\ZZ}{\mathbb{Z}}

\newcommand{\GL}{\mathrm{GL}}
\newcommand{\Hom}{\mathrm{Hom}}
\newcommand{\Perv}{\mathrm{Perv}}
\newcommand{\Rep}{\mathrm{Rep}}
\newcommand{\Groth}{\underline{\mathrm{Groth}}}
\newcommand{\Spr}{\underline{\mathrm{Spr}}}
\newcommand{\Gr}{\mathrm{Gr}}
\newcommand{\sm}{\mathrm{sm}}
\newcommand{\rs}{\mathrm{rs}}

\newcommand{\Iff}{\Longleftrightarrow}
\newcommand{\relmiddle}[1]{\mathrel{}\middle#1\mathrel{}}
\DeclareMathOperator{\Spec}{Spec}

\makeatletter
\DeclareRobustCommand{\dwt}[1]{{%
\mathpalette\double@widetilde{#1}%
}}
\DeclareRobustCommand{\double@widetilde}[2]{%
\sbox\z@{$\m@th#1\widetilde{#2}$}%
\ht\z@=.85\ht\z@
\widetilde{\box\z@}%
}
\makeatother

\newcommand{\prism}[9]{
\def\tempa{#1}%
\def\tempb{#2}%
\def\tempc{#3}%
\def\tempd{#4}%
\def\tempe{#5}%
\def\tempf{#6}%
\def\tempg{#7}%
\def\temph{#8}%
\def\tempi{#9}%
\prismcontinued
}
\newcommand{\prismcontinued}[9]{
\def\tempj{#1}%
\def\tempk{#2}%
\def\templ{#3}%
\def\tempm{#4}%
\def\tempn{#5}%
\def\tempo{#6}%
\def\tempp{#7}%
\def\tempq{#8}%
\def\tempr{#9}%
\prismcontinuedcontinued
}
\newcommand{\prismcontinuedcontinued}[2]{
\xymatrix{
\tempa\ar[rr]^{\tempb}\ar[rd]^{\tempd}\ar[dd]_{\tempq}\ar@{}[rrdd]|{\wonlg{\tempr}}\ar@{}[rrrd]|{\italicrd{\wong{\tempe}}}
\ar@<5ex>@{}[dd]|{\wong{\tempg}}
&&\tempc\ar[rd]^{\tempf}\ar@{.>}[dd]^(.25){#1}|\hole \ar@<5ex>@{}[dd]|{\wonlg{#2}}&\\
&\temph\ar[rr]_{\tempi}\ar[ld]_{\tempk}\ar@{}[rd]|{\italicld{\wong{\templ}}}&&\tempj\ar[ld]^{\tempm}\\
\tempn\ar[rr]_{\tempo}&&\tempp&
}
}

\newcommand{\opprismcontinuedcontinued}[2]{
\xymatrix{
\tempa\ar[rr]^{\tempb}\ar@{<-}[rd]^{\tempd}\ar@{<-}[dd]_{\tempq}\ar@{}[rrdd]|{\wonlg{\tempr}}\ar@{}[rrrd]|{\italicrd{\wong{\tempe}}}
\ar@<5ex>@{}[dd]|{\wong{\tempg}}
&&\tempc\ar@{<-}[rd]^{\tempf}\ar@{<.}[dd]^(.25){#1}|\hole \ar@<5ex>@{}[dd]|{\wonlg{#2}}&\\
&\temph\ar[rr]_{\tempi}\ar@{<-}[ld]_{\tempk}\ar@{}[rd]|{\italicld{\wong{\templ}}}&&\tempj\ar@{<-}[ld]^{\tempm}\\
\tempn\ar[rr]_{\tempo}&&\tempp&
}
}

\newcommand{\opopprismcontinuedcontinued}[2]{
\xymatrix{
\tempa\ar@{<-}[rr]^{\tempb}\ar@{<-}[rd]^{\tempd}\ar@{<-}[dd]_{\tempq}\ar@{}[rrdd]|{\wonlg{\tempr}}\ar@{}[rrrd]|{\italicrd{\wong{\tempe}}}
\ar@<5ex>@{}[dd]|{\wong{\tempg}}
&&\tempc\ar@{<-}[rd]^{\tempf}\ar@{<.}[dd]^(.25){#1}|\hole \ar@<5ex>@{}[dd]|{\wonlg{#2}}&\\
&\temph\ar@{<-}[rr]_{\tempi}\ar@{<-}[ld]_{\tempk}\ar@{}[rd]|{\italicld{\wong{\templ}}}&&\tempj\ar@{<-}[ld]^{\tempm}\\
\tempn\ar@{<-}[rr]_{\tempo}&&\tempp&
}
}

\newcommand{\horopprismcontinuedcontinued}[2]{
\xymatrix{
\tempa\ar@{<-}[rr]^{\tempb}\ar[rd]^{\tempd}\ar[dd]_{\tempq}\ar@{}[rrdd]|{\wonlg{\tempr}}\ar@{}[rrrd]|{\italicrd{\wong{\tempe}}}
\ar@<5ex>@{}[dd]|{\wong{\tempg}}
&&\tempc\ar[rd]^{\tempf}\ar@{.>}[dd]^(.25){#1}|\hole \ar@<5ex>@{}[dd]|{\wonlg{#2}}&\\
&\temph\ar@{<-}[rr]_{\tempi}\ar[ld]_{\tempk}\ar@{}[rd]|{\italicld{\wong{\templ}}}&&\tempj\ar[ld]^{\tempm}\\
\tempn\ar@{<-}[rr]_{\tempo}&&\tempp&
}
}

\newcommand{\leqprismcontinuedcontinued}[2]{
\xymatrix@!C{
\tempa\ar[rr]^{\tempb}\ar@{=}[rd]^{\tempd}\ar@{=}[dd]_{\tempq}\ar@{}[rrdd]|{\wonlg{\tempr}}\ar@{}[rrrd]|{\italicrd{\wong{\tempe}}}
\ar@<5ex>@{}[dd]|{\wong{\tempg}}
&&\tempc\ar[rd]^{\tempf}\ar@{.>}[dd]^(.25){#1}|\hole \ar@<5ex>@{}[dd]|{\wonlg{#2}}&\\
&\temph\ar[rr]_{\tempi}\ar@{=}[ld]_{\tempk}\ar@{}[rd]|{\italicld{\wong{\templ}}}&&\tempj\ar[ld]^{\tempm}\\
\tempn\ar[rr]_{\tempo}&&\tempp&
}
}

\newcommand{\leqopprismcontinuedcontinued}[2]{
\xymatrix@!C{
\tempa\ar[rr]^{\tempb}\ar@{=}[rd]^{\tempd}\ar@{=}[dd]_{\tempq}\ar@{}[rrdd]|{\wonlg{\tempr}}\ar@{}[rrrd]|{\italicrd{\wong{\tempe}}}
\ar@<5ex>@{}[dd]|{\wong{\tempg}}
&&\tempc\ar@{<-}[rd]^{\tempf}\ar@{<.}[dd]^(.25){#1}|\hole \ar@<5ex>@{}[dd]|{\wonlg{#2}}&\\
&\temph\ar[rr]_{\tempi}\ar@{=}[ld]_{\tempk}\ar@{}[rd]|{\italicld{\wong{\templ}}}&&\tempj\ar@{<-}[ld]^{\tempm}\\
\tempn\ar[rr]_{\tempo}&&\tempp&
}
}

\newcommand{\reqprismcontinuedcontinued}[2]{
\xymatrix@!C{
\tempa\ar[rr]^{\tempb}\ar[rd]^{\tempd}\ar[dd]_{\tempq}\ar@{}[rrdd]|{\wonlg{\tempr}}\ar@{}[rrrd]|{\italicrd{\wong{\tempe}}}
\ar@<5ex>@{}[dd]|{\wong{\tempg}}
&&\tempc\ar@{=}[rd]^{\tempf}\ar@2{.}[dd]^(.25){#1}|\hole \ar@<5ex>@{}[dd]|{\wonlg{#2}}&\\
&\temph\ar[rr]_{\tempi}\ar[ld]_{\tempk}\ar@{}[rd]|{\italicld{\wong{\templ}}}&&\tempj\ar@{=}[ld]^{\tempm}\\
\tempn\ar[rr]_{\tempo}&&\tempp&
}
}

\newcommand{\reqopopprismcontinuedcontinued}[2]{
\xymatrix@!C{
\tempa\ar@{<-}[rr]^{\tempb}\ar@{<-}[rd]^{\tempd}\ar@{<-}[dd]_{\tempq}\ar@{}[rrdd]|{\wonlg{\tempr}}\ar@{}[rrrd]|{\italicrd{\wong{\tempe}}}
\ar@<5ex>@{}[dd]|{\wong{\tempg}}
&&\tempc\ar@{=}[rd]^{\tempf}\ar@2{.}[dd]^(.25){#1}|\hole \ar@<5ex>@{}[dd]|{\wonlg{#2}}&\\
&\temph\ar[rr]_{\tempi}\ar@{<-}[ld]_{\tempk}\ar@{}[rd]|{\italicld{\wong{\templ}}}&&\tempj\ar@{=}[ld]^{\tempm}\\
\tempn\ar[rr]_{\tempo}&&\tempp&
}
}

\newcommand{\cubecontinuedcontinued}[8]{
\xymatrix@!C{
\tempa\ar[rr]^{\tempb}\ar[rd]^{\tempd}\ar[dd]_{\tempj}
\ar@{}[rrrd]|{\italicrd{\wong{\tempe}}}\ar@{}[rrdd]|{\wonlg{\tempk}}\ar@{}@<1.5ex>[rddd]|{\italicside{\wong{\tempm}}}
&&\tempc\ar[rd]^{\tempf}\ar@{.>}[dd]^(.75){\templ}\ar@{}@<1.5ex>[rddd]|{\italicside{\wonlg{\tempn}}}&\\
&\tempg\ar[rr]^(.25){\temph}\ar[dd]_(.25){\tempo}\ar@{}[rrdd]|{\wong{\tempp}}&&\tempi\ar[dd]^{\tempq}\\
\tempr\ar@{.>}[rr]^(.75){#1}\ar[rd]^{#3}\ar@{}[rrrd]|{\italicrd{\wonlg{#4}}}&&#2\ar@{.>}[rd]^{#5}&\\
&#6\ar[rr]^{#7}&&#8
}
}

\newcommand{\opcubecontinuedcontinued}[8]{
\xymatrix@!C{
\tempa\ar[rr]^{\tempb}\ar[rd]^{\tempd}\ar@{<-}[dd]_{\tempj}
\ar@{}[rrrd]|{\italicrd{\wong{\tempe}}}\ar@{}[rrdd]|{\wonlg{\tempk}}\ar@{}@<1.5ex>[rddd]|{\italicside{\wong{\tempm}}}
&&\tempc\ar[rd]^{\tempf}\ar@{<.}[dd]^(.75){\templ}\ar@{}@<1.5ex>[rddd]|{\italicside{\wonlg{\tempn}}}&\\
&\tempg\ar[rr]^(.25){\temph}\ar@{<-}[dd]_(.25){\tempo}\ar@{}[rrdd]|{\wong{\tempp}}&&\tempi\ar@{<-}[dd]^{\tempq}\\
\tempr\ar@{.>}[rr]^(.75){#1}\ar[rd]^{#3}\ar@{}[rrrd]|{\italicrd{\wonlg{#4}}}&&#2\ar@{.>}[rd]^{#5}&\\
&#6\ar[rr]^{#7}&&#8
}
}

\newcommand{\opopcubecontinuedcontinued}[8]{
\xymatrix@!C{
\tempa\ar[rr]^{\tempb}\ar@{<-}[rd]^{\tempd}\ar@{<-}[dd]_{\tempj}
\ar@{}[rrrd]|{\italicrd{\wong{\tempe}}}\ar@{}[rrdd]|{\wonlg{\tempk}}\ar@{}@<1.5ex>[rddd]|{\italicside{\wong{\tempm}}}
&&\tempc\ar@{<-}[rd]^{\tempf}\ar@{<.}[dd]^(.75){\templ}\ar@{}@<1.5ex>[rddd]|{\italicside{\wonlg{\tempn}}}&\\
&\tempg\ar[rr]^(.25){\temph}\ar@{<-}[dd]_(.25){\tempo}\ar@{}[rrdd]|{\wong{\tempp}}&&\tempi\ar@{<-}[dd]^{\tempq}\\
\tempr\ar@{.>}[rr]^(.75){#1}\ar@{<-}[rd]^{#3}\ar@{}[rrrd]|{\italicrd{\wonlg{#4}}}&&#2\ar@{<.}[rd]^{#5}&\\
&#6\ar[rr]^{#7}&&#8
}
}

\newcommand{\horopcubecontinuedcontinued}[8]{
\xymatrix@!C{
\tempa\ar@{<-}[rr]^{\tempb}\ar[rd]^{\tempd}\ar[dd]_{\tempj}
\ar@{}[rrrd]|{\italicrd{\wong{\tempe}}}\ar@{}[rrdd]|{\wonlg{\tempk}}\ar@{}@<1.5ex>[rddd]|{\italicside{\wong{\tempm}}}
&&\tempc\ar[rd]^{\tempf}\ar@{.>}[dd]^(.75){\templ}\ar@{}@<1.5ex>[rddd]|{\italicside{\wonlg{\tempn}}}&\\
&\tempg\ar@{<-}[rr]^(.25){\temph}\ar[dd]_(.25){\tempo}\ar@{}[rrdd]|{\wong{\tempp}}&&\tempi\ar[dd]^{\tempq}\\
\tempr\ar@{<.}[rr]^(.75){#1}\ar[rd]^{#3}\ar@{}[rrrd]|{\italicrd{\wonlg{#4}}}&&#2\ar@{.>}[rd]^{#5}&\\
&#6\ar@{<-}[rr]^{#7}&&#8
}
}

\newcommand{\leqcubecontinuedcontinued}[8]{
\xymatrix@!C{
\tempa\ar[rr]^{\tempb}\ar@{=}[rd]^{\tempd}\ar@{=}[dd]_{\tempj}
\ar@{}[rrrd]|{\italicrd{\wong{\tempe}}}\ar@{}[rrdd]|{\wonlg{\tempk}}\ar@{}@<1.5ex>[rddd]|{\italicside{\wong{\tempm}}}
&&\tempc\ar[rd]^{\tempf}\ar@{.>}[dd]^(.75){\templ}\ar@{}@<1.5ex>[rddd]|{\italicside{\wonlg{\tempn}}}&\\
&\tempg\ar[rr]^(.25){\temph}\ar@{=}[dd]_(.25){\tempo}\ar@{}[rrdd]|{\wong{\tempp}}&&\tempi\ar[dd]^{\tempq}\\
\tempr\ar@{.>}[rr]^(.75){#1}\ar@{=}[rd]^{#3}\ar@{}[rrrd]|{\italicrd{\wonlg{#4}}}&&#2\ar@{.>}[rd]^{#5}&\\
&#6\ar[rr]^{#7}&&#8
}
}

\newcommand{\leqveropcubecontinuedcontinued}[8]{
\xymatrix@!C{
\tempa\ar[rr]^{\tempb}\ar@{=}[rd]^{\tempd}\ar@{=}[dd]_{\tempj}
\ar@{}[rrrd]|{\italicrd{\wong{\tempe}}}\ar@{}[rrdd]|{\wonlg{\tempk}}\ar@{}@<1.5ex>[rddd]|{\italicside{\wong{\tempm}}}
&&\tempc\ar[rd]^{\tempf}\ar@{<.}[dd]^(.75){\templ}\ar@{}@<1.5ex>[rddd]|{\italicside{\wonlg{\tempn}}}&\\
&\tempg\ar[rr]^(.25){\temph}\ar@{=}[dd]_(.25){\tempo}\ar@{}[rrdd]|{\wong{\tempp}}&&\tempi\ar@{<-}[dd]^{\tempq}\\
\tempr\ar@{.>}[rr]^(.75){#1}\ar@{=}[rd]^{#3}\ar@{}[rrrd]|{\italicrd{\wonlg{#4}}}&&#2\ar@{.>}[rd]^{#5}&\\
&#6\ar[rr]^{#7}&&#8
}
}

\title{\Large{Geometric Satake equivalence in mixed characteristic and Springer correspondence}}
\author{Katsuyuki Bando\footnote{Katsuyuki Bando, the University of Tokyo, Graduate School of Mathematical Sciences, Japan, \texttt{kbando@ms.u-tokyo.ac.jp}}}

\begin{document}
\maketitle
 \abstract{The geometric Satake equivalence and the Springer correspondence are closely related when restricting to small representations of the Langlands dual group.
We prove this result for \'{e}tale sheaves, including the case of the mixed characteristic affine Grassmannian, assuming sufficiently large ramification.
In this process, we construct a canonical partial isomorphism between a mixed characteristic affine Grassmannian under the assumption of sufficiently large ramification and an equal characteristic one.}
\section{Introduction}
The geometric Satake equivalence is an equivalence between the category of perverse sheaves on the affine Grassmannian of a reductive algebraic group and the category of representations of its Langlands dual group.
This equivalence, relating a geometric category to a representation-theoretic category, is a fundamental tool in geometric representation theory and also in number theory.

Several versions of this equivalence exist, including one involving the affine Grassmannian defined using the equal characteristic field $k((t))$ with $k$ a field, and one involving the mixed characteristic field $\QQ_p$.
The former, which we mainly call the geometric Satake equivalence in equal characteristic in this paper, was proved in \cite{MV}, and the latter, the geometric Satake in mixed characteristic, was proved in \cite{Zhumixed} for $\Qlb$-coefficients, and in \cite{Yu} for integral or mod $\ell$ coefficients.
Furthermore, Fargues--Scholze \cite{FS} established a geometric Satake equivalence for the $B_{dR}$-affine Grassmannian in the framework of perfectoid spaces and diamonds.
This construction plays a foundational role in their geometrization of the local Langlands correspondence.

In this paper, we first construct a canonical isomorphism between a subscheme of Zhu's mixed characteristic affine Grassmannian (under the assumption of sufficiently large ramification) and a subscheme of an equal characteristic one in \S \ref{scn:defofMG}:
\begin{thm} \label{thm:Grisomgeneralintro}
Let $F_1,F_2$ be two complete discrete valuation fields with the same residue field $k$.
Let $(G_{\ZZ},B_{\ZZ},T_{\ZZ})$ be a triple of a split reductive group over $\ZZ$, its Borel group and its maximal torus such that $T\subset B$.
Put $(G_1,B_1,T_1)=(G_{\ZZ},B_{\ZZ},T_{\ZZ})\otimes_{\ZZ}\cal{O}_{F_1}$ and $(G_2,B_2,T_2)=(G_{\ZZ},B_{\ZZ},T_{\ZZ})\otimes_{\ZZ}\cal{O}_{F_2}$.
Let $\mu$ be an element in $\mathbb{X}_{\bullet}^+(T_{\ZZ})$.
Then there exists a constant $N:=N_{G_{\ZZ},\mu}\in \mathbb{Z}_{>0}$ depending only on the isomorphism class of $G_{\ZZ}$ and $\mu$ such that if $v_{F_1}(p)\geq N$ and $v_{F_2}(p)\geq N$, then there is a ``canonical'' isomorphism
\begin{align*}
\alpha_{G_{\ZZ}, \mu,\cal{O}_1,\cal{O}_2}\colon \Gr_{G_1,\leq \mu} \cong \Gr_{G_2,\leq \mu}.
\end{align*}
Moreover, the Schubert cells $\Gr_{G_1,\lambda}$ and $\Gr_{G_2,\lambda}$ correspond under $\alpha_{G_{\ZZ}, \mu,\cal{O}_1,\cal{O}_2}$ for any $\lambda\leq \mu$.
\end{thm}
The meaning of ``canonical'' is explained after Theorem \ref{thm:Grisomgeneral}.

As an application, we will prove the relationship between the geometric Satake equivalence and the Springer correspondence for \'{e}tale sheaves, including the case of the mixed characteristic, assuming a sufficiently large ramification.
For a complex reductive group $G$, this is studied in \cite{AHR}.
The precise statement is as follows:

Let $k$ be an algebraically closed field of characteristic $p$ and let $G$ be a reductive group over $\cal{O}$, where $\cal{O}$ is $k[[t]]$ or a totally ramified finite extension of the ring of Witt vectors $W(k)$.
Assume that $p$ is a good prime for $G$.

We denote the affine Grassmannian by $\Gr_G$.
Let $\cal{N}_G$ be the nilpotent cone in the Lie algebra of $\bar{G}$ (the reduction of $G$).
Let $\ell$ be a prime different from $p$.
Consider the following four functors (see \S \ref{scn:GSSC} for precise definitions):
\begin{enumerate}
\item[(1)] The restriction of the geometric Satake equivalence $\scr{S}\colon \Perv_{L^+G}(\Gr,\Qlb)\to \Rep(\check{G},\Qlb)$ to small representations of the Langlands dual group:
\[
\scr{S}_G^{\sm}\colon \Perv_{L^+G}(\Gr^{\sm}_G,\Qlb)\to \Rep(\check{G},\Qlb)_{\sm}.
\]
\item[(2)] The Springer correspondence
\[
\bb{S}_G=\Hom(\Spr,-)\colon \Perv_{\bar{G}}(\cal{N}_G^{p^{-\infty}},\Qlb)\cong \Perv_{\bar{G}}(\cal{N}_G,\Qlb)\to \Rep(W_G,\Qlb)
\]
where $\Spr\in \Perv_{\bar{G}}(\cal{N}_G,\Qlb)$ is the Springer sheaf and $W_G$ is the (finite) Weyl group of $G$.
\item[(3)] By taking the zero weight space of a representation of $\check{G}$ and tensoring it with the sign character $\varepsilon\colon W_G\to \Qlb^{\times}$, we obtain a functor
\[
\Phi_{\check{G}}\colon \Rep(\check{G},\Qlb)_{\sm}\to \Rep(W_G,\Qlb).
\]
\item[(4)] Assuming sufficiently large ramification of $\cal{O}$, there is an open dense subspace $\cal{M}_G\subset \Gr^{\sm}_G$ and $\pi_G\colon \cal{M}_G\to \cal{N}_G^{p^{-\infty}}$.
The map $\pi_G$ is the perfection of a finite map under the assumption that $p$ is good for $G$.
We obtain a functor
\[
\Psi_{G}:=(\pi_G)_*\circ (j_G)^!\colon \Perv_{L^+G}(\Gr^{\sm}_G,\Qlb) \to \Perv_{\bar{G}}(\cal{N}^{p^{-\infty}}_G,\Qlb)
\]
where $j_G\colon \cal{M}_G\hookrightarrow \Gr^{\sm}_G$ is the inclusion.
\end{enumerate}
We have the diagram
\[
\xymatrix{
\Perv_{L^+G}(\Gr^{\sm}_G,\Qlb)\ar[r]^-{\scr{S}^\sm_G}\ar[d]_-{\Psi_G}&\Rep(\check{G},\Qlb)_{\sm}\ar[d]^-{\Phi_{\check{G}}}\\
\Perv_{\bar{G}}(\cal{N}^{p^{-\infty}}_G,\Qlb)\ar[r]_{\mathbb{S}_G}&\Rep(W_G,\Qlb).\\
}
\]
The theorem is the following:
\begin{thm}\label{thm:mainfirst}
Assume sufficiently large ramification of $\cal{O}$ and that $p$ is good for $G$.
Then, there is a canonical isomorphism of functors:
\[
\Phi_{\check{G}} \circ \scr{S}^\sm_G \Iff \mathbb{S}_G \circ \Psi_G
\]
\end{thm}
The meaning of ``sufficiently large ramification'' is explained in the proof of Theorem \ref{thm:Grisomgeneral}.
For many parts of the proof, the same method as \cite{AHR} can be used:
We use the same method as \cite{AHR} to reduce to the case where $G$ is semisimple of rank 1, which is explained in \cite[\S 3]{AHR} or \S \ref{scn:plan} in the present paper.

However, the construction of the open subset $\cal{M}_G\subset \Gr^{\sm}_G$ in \cite[\S 2.6]{AHR} does not work in the mixed characteristic case.
In fact, the subset $\cal{M}_G$ in equal characteristic is defined as the intersection of $\Gr^{\sm}_G$ with a certain $G(k[t^{-1}])$-orbit $\Gr_{G,0}^-\subset \Gr_G$.
There is no subalgebra in mixed characteristic corresponding to $k[t^{-1}]\subset k((t))$.
That is why we use Theorem \ref{thm:Grisomgeneralintro}.
Thanks to this theorem, we can define $\cal{M}_G$ in the mixed characteristic case as the pullback of $\cal{M}_G$ in the equal characteristic case.

Moreover, we cannot use the method in \cite{AHR} for the proof in the case where $G$ is semisimple of rank 1.
This is because the global version of the affine Grassmannian used in \cite[\S 8]{AHR} does not exist in mixed characteristic.
Instead, we use a method in \cite[\S 4.1]{AH}.

Additionally, we remark that we use the result \cite[Proposition VI.9.6]{FS} in Fargues--Scholze's paper in order to get a monoidal structure of the restriction functor, called the constant term functor, from the Satake category of $G$ to that of its Levi subgroup.

We believe that the method used in Theorem \ref{thm:Grisomgeneralintro} can also be used for a partial isomorphism between an object in equal characteristic and its mixed characteristic version other than affine Grassmannians.
We also hope that the result in this paper will be linked to the coherent Springer theory explained in \cite{BCHN}.
\subsection*{Acknowledgement}
The author is really grateful to his advisor Naoki Imai for useful discussions and comments.
He also thanks the referees for their careful reading of the paper and for their valuable suggestions and comments.
This work is supported by JSPS KAKENHI Grant Number 22J22110 and 25KJ0189.
The author was supported by Grant-in-Aid for JSPS Fellows and FMSP, WINGS Program, the University of Tokyo.
\section{Notations}
 Let $k$ be an algebraically closed field of characteristic $p>0$.
Fix a prime number $\ell\neq p$.
In this paper, all sheaves are $\Qlb$-sheaves for the \'etale topology.
If $X$ is an (ind-)$k$-variety or the perfection of it, we write $D^b(X,\Qlb)$ or simply $D^b(X)$ for the $\ell$-adic bounded constructible derived category of $X$.
If $H$ is a connected (pro-)algebraic group over $k$ or the perfection of it acting on $X$,  we write $\Perv_H(X,\Qlb)$ for the full abelian subcategory (of $D^b(X,\Qlb)$) consisting of $H$-equivariant perverse sheaves.
We write $D^b_H(X,\Qlb)$ for the $H$-equivariant derived category defined in \cite{BL}, and $\Perv'_H(X,\Qlb)$ for its core with respect to the perverse $t$-structure.
Note that the forgetful functor $\sf{For}\colon D^b_H(X,\Qlb) \to D^b(X,\Qlb)$ induces the equivalence $\sf{For}\colon \Perv'_H(X,\Qlb)\to \Perv_H(X,\Qlb)$ (see \cite[Proposition 2.5.3]{BL}).

For a $k$-algebra $R$, let $W(R)$ denote the ring of Witt vectors, and $W_h(R)$ the truncated Witt vectors of length $h$.
For $a\in R$, write $[a]$ for the Teichm\"uller lift of $a$.
Let $F$ be $k((t))$ or a totally ramified finite extension of $W(k)[1/p]$.
Let $\cal{O}$ be the ring of integers of $F$.
We take a uniformizer $\varpi$ of $\cal{O}$ (if $F=k((t))$, set $\varpi=t$).\\
For a $k$-algebra $R$, we write
\begin{align*}
W_{\cal{O}}(R)&=W(R)\otimes_{W(k)} \cal{O},\\
\cal{O}_R&=\begin{cases}
W_{\cal{O}}(R)&\text{if $F$ is a totally ramified finite extension of $W(k)[1/p]$,}\\
R[[t]]&\text{if $F=k((t))$}
\end{cases}
\end{align*}
and 
\begin{align*}
W_{\cal{O},h}(R)&=W_{\cal{O}}(R)/\varpi^h,\\
\cal{O}_{R,h}&=\cal{O}_R/\varpi^h.
\end{align*}
Throughout this paper, $G$ is a connected reductive group over $\cal{O}$.
Since any reductive group is \'{e}tale locally split, and since $k$ is algebraically closed, it follows that $G$ is split.
We choose a Borel subgroup $B$ of $G$ and a maximal torus $T$ of $B$.
Let $U$ be the unipotent radical of $B$.\\
We write $\bar{G},\bar{B},\bar{U},\bar{T}$ for the reduction modulo $\varpi$ of $G,B,U,T$, respectively and $\fk{g},\fk{b},\fk{u},\fk{t}$ for the Lie algebra of $\bar{B},\bar{U},\bar{T}$, respectively.
Let $\cal{N}_G\subset \fk{g}$ be the nilpotent cone, and $\cal{N}_G^{p^{-\infty}}$ its perfection.
Let $W_G$ be the Weyl group $N_G(T)/T$.

Throughout this paper, $P$ denotes a parabolic subgroup of $G$ containing $B$, and $L$ denotes its Levi subgroup (containing $T$).
We have the Levi decomposition $P=LU_P$ where $U_P$ is the unipotent radical of $P$.
Sometimes $B_L$ denotes $B\cap L$.
The group $B_L$ is a Borel subgroup of $L$ containing $T$.

Let $\mathbb{X}^{\bullet}(T)$ and $\mathbb{X}_{\bullet}(T)$ be the character and cocharacter lattice of $T$.
Let $\mathbb{X}_{\bullet}^+(T)\subset \mathbb{X}_{\bullet}(T)$ be the subset of dominant cocharacters with respect to $B$.

From \S \ref{scn:plan}, we use the 2-categorical notions explained in \cite[\S A]{AHR}, such as ``pasting diagrams'', ``commutative prisms'' and ``commutative cubes''.
In these diagrams, we also use the labels (Co), (BC), (For), (Tr), (FF), (FT) for the corresponding natural morphisms, defined in the same way as \cite[\S B]{AHR}.
\section{Preliminaries on affine Grassmannians and geometric Satake equivalences}\label{scn:preliminaries}
\subsection{Affine Grassmannians}\label{ssc:affgrass}
Let $H$ be a smooth affine group scheme over $\cal{O}$.
We consider the following presheaves on the category of perfect affine $k$-schemes
\begin{align*}
L^+H(R)&=H(\cal{O}_R),\\
L^hH(R)&=H(\cal{O}_{R,h}),\\
LH(R)&=H(\cal{O}_R[1/\varpi]).
\end{align*}
By \cite[Proposition 1.3.2]{Zhueq} or \cite[1.1]{Zhumixed}, $L^+H$ and $L^hH$ are represented by pfp (perfectly of finite presentation) perfect group schemes over $k$, and $LH$ is represented by an ind perfect scheme over $k$.
We also write $L^+H^{(h)}$ for the $h$-th congruence group (i.e. the kernel of the natural map $L^+H\to L^hH$).

The affine Grassmannian of $H$ is the perfect space $\Gr_H$ defined as a quotient sheaf
\[
\Gr_H:=LH/L^+H.
\]
By \cite[Theorem 1.1]{BS}, $\Gr_H$ is representable by the inductive limit of perfections of quasi-projective varieties over $k$, along closed immersions.
If $H$ is a reductive group scheme, then $\Gr_H$ is represented by the inductive limit of perfections of projective varieties over $k$, along closed immersions.
\begin{rmk}\label{rem:nonperfectversion}
If $F=k((t))$, there are non-perfect versions of the above spaces (see \cite{Zhueq}).
We write $\Gr'_H$ for this canonical deperfection of $\Gr_H$.
All the results for equal characteristic in this paper hold for this version by the same arguments.
However, there is no canonical deperfection of $\Gr_G$ in mixed characteristic, see \cite[Remark B.6]{Zhumixed}.
That is why we work with perfect schemes.
\end{rmk}

For $\lambda \in \mathbb{X}_{\bullet}(T)$, let $\varpi^{\lambda}\in \Gr_G(k)$ be the image of $\varpi\in F^{\times} =L\mathbb{G}_m(k)$ under the map
\[
L\mathbb{G}_m \overset{L\lambda}{\to} LT\to LG\to \Gr_G.
\]
For $\mu \in \mathbb{X}^+_{\bullet}(T)$, define $\Gr_{G,\mu}$ as the $L^+G$-orbit of $\varpi^{\mu}$, and let
\[
\Gr_{G,\leq \mu}=\bigcup_{\mu'\leq \mu}\Gr_{G,\mu'}.
\]
The space $\Gr_{G,\leq \mu}$ is a closed subspace of $\Gr_G$, and $\Gr_{G,\mu}$ is a dense open subspace of $\Gr_{G,\leq \mu}$.
We call $\Gr_{G,\leq \mu}$ the Schubert variety corresponding to $\mu$, and $\Gr_{G,\mu}$ the Schubert cell corresponding to $\mu$.

For $\lambda \in \mathbb{X}_{\bullet}(T)$ define $S_{G,\lambda}=S_{\lambda}$ as the $LU$-orbit of $\varpi^{\lambda}$, and let 
\[
S_{\leq \lambda}=\bigcup_{\lambda'\leq \lambda}S_{\lambda'}.
\]
The space $S_{\leq \lambda}$ is a closed subspace of $\Gr_G$, and $S_{\lambda}$ is an open subspace of $S_{\leq \lambda}$.
The space $S_{\lambda}$ is called a semi-infinite orbit; it is the attractor locus of a certain torus action on $\Gr_G$ (see \cite[\S 2.2]{Zhumixed} for details).
The natural map
\[
\Gr_B\to \Gr_G
\]
can be identified with
\[
\coprod_{\lambda \in \bb{X}_{\bullet}(T)} S_{\lambda}\to \Gr_G.
\]
By the Iwasawa decomposition, this map is bijective on points.

Let $\check{\Phi}$ be the set of roots of $(\check{G},\check{T})$.
The Weyl group of $\check{G}$ is identified with the Weyl group of $G$, denoted by $W_G$.
The group $W_G$ acts on $\mathbb{X}_{\bullet}(T)=\mathbb{X}^{\bullet}(\check{T})$.
\begin{defi}
A character $\mu\in \mathbb{X}^{\bullet}(\check{T})$ is said to be small for $\check{G}$ if
\[
\begin{cases}
\mu\in \mathbb{Z}\check{\Phi},\\
\text{convex hull of $W_G\cdot \mu$ does not contain any element of $\{2\check{\alpha}\mid \check{\alpha}\in \check{\Phi}\}.$}
\end{cases}
\]
The closed subspace $\Gr^{\sm}\subset \Gr$ is defined as the union of $\Gr_{\mu}$ for small $\mu\in \mathbb{X}^+_{\bullet}(T)$.
\end{defi}
\subsection{Geometric Satake equivalence}
By \cite{MV} or \cite{Zhumixed}, $\Perv_{L^+G}(\Gr_G,\Qlb)$ has a monoidal structure under the convolution product $\star$, and the functor 
\[
\mathsf{F}_G:=H^*(\Gr_G,-)\colon \Perv_{L^+G}(\Gr_G,\Qlb)\to \rm{Vect}_{\Qlb}
\]
is monoidal.
Put
\[
\check{G}:= \mathrm{Aut}^{\otimes}(\mathsf{F}_G).
\]
Here $\mathrm{Aut}^{\otimes}(\mathsf{F}_G)$ is the Tannakian group of the fiber functor $\mathsf{F}_G$.
Then $\check{G}$ is isomorphic to the Langlands dual group over $\Qlb$ of $G$, and $\mathsf{F}_G$ gives rise to an equivalence of monoidal categories
\[
\scr{S}_G\colon \ \Perv_{L^+G}(\Gr_G,\Qlb)\to \Rep(\check{G},\Qlb).
\]
This is called a geometric Satake equivalence.
More explicitly, the intersection cohomology sheaf $\mathrm{IC}_{\mu}$ on $\Gr_{\leq \mu}$ corresponds to an irreducible representation with highest weight $\mu$.
\section{Partial isomorphism between affine Grassmannians}\label{scn:defofMG}
\subsection{Preliminaries on affine Grassmannian for $\GL_n$}
In this subsection, we will introduce some definitions and fundamental results on the affine Grassmannian for $\GL_n$.

Recall that there is another interpretation of the affine Grassmannian.
Namely, 
\[
\Gr_G(R)=\left\{(\cal{E},\beta)\relmiddle| \begin{array}{l}
\text{$\cal{E}$ is a $G$-torsor on $\Spec \cal{O}_R$,}\\
\text{$\beta\colon \cal{E}|_{\Spec \cal{O}_{R}[1/\varpi]}\to \cal{E}_0|_{\Spec \cal{O}_{R}[1/\varpi]}$ is an isomorphism.},
\end{array}\right\}
\]
where $\cal{E}_0$ is the trivial $G$-torsor on $\Spec \cal{O}_R$.
In particular, if $G=\GL_n$, we have
\begin{align} \label{Grmoduli}
\Gr_G(R)=\left\{(\cal{E},\beta)\relmiddle| \begin{array}{l}
\text{$\cal{E}$ is a projective $\cal{O}_R$-module of rank $n$,}\\
\text{$\beta\colon \cal{E}[1/\varpi]\to \cal{E}_0[1/\varpi]$ is an isomorphism}
\end{array}\right\},
\end{align}
where $\cal{E}_0=\cal{O}_R^n$.

For finite projective $\cal{O}_R$-modules $\cal{E}_1$ and $\cal{E}_2$, an isomorphism $\beta\colon \cal{E}_2[1/\varpi]\to \cal{E}_1[1/\varpi]$ is called a quasi-isogeny.
We write this as $\beta\colon \cal{E}_2\dashrightarrow \cal{E}_1$.
It is called an isogeny if it is induced by a genuine map $\cal{E}_2\to \cal{E}_1$.

Recall that when $G=\GL_n$, we can make the identifications
\begin{align*}
\bb{X}_{\bullet}(T)&=\ZZ^n,\\
\bb{X}_{\bullet}^+(T)&=\{(m_1,\ldots,m_n)\in \ZZ^n\mid m_1\geq \cdots \geq m_n\}
\end{align*}
and the partial order on $\bb{X}_{\bullet}(T)$ can be described as follows:
\begin{align*}
&(m_1,\ldots,m_n)\leq (l_1,\ldots,l_n)\text{ if and only if}\\
&\begin{cases}
m_1+\cdots +m_j\leq l_1+\cdots +l_j&(1\leq j\leq n),\\
m_1+\cdots +m_n= l_1+\cdots +l_n&
\end{cases}
\end{align*}

Let $R$ be a perfect field over $k$.
Then for a quasi-isogeny $\beta\colon \cal{E}_1 \dashrightarrow \cal{E}_2$, there exists a basis $(e_1,\ldots,e_n)$ of $\cal{E}_1$ and a basis $(f_1,\ldots,f_n)$ of $\cal{E}_2$ such that
\[
\beta(e_i)=\varpi^{m_i}f_i
\]
and $(m_1,m_2,\ldots,m_n)\in \bb{X}_{\bullet}^+(T)$.
Then we write
\[
\mathrm{Inv}(\beta)=(m_1,\ldots,m_n)
\]
and call it the relative position of $\beta$.

Let $R$ be a general perfect $k$-algebra, and let $\beta\colon \cal{E}_1 \dashrightarrow \cal{E}_2$ be a quasi-isogeny.
For $x\in \Spec R$, we write $\beta_x:=\beta\otimes_{\cal{O}_R}\cal{O}_{k(x)}$.

Then for $\mu\in \bb{X}_{\bullet}^+(T)$, the Schubert variety and the Schubert cell can be described as
\begin{align*}
\Gr_{G,\leq \mu}&=\{(\cal{E},\beta)\in \Gr_G\mid \mathrm{Inv}(\beta_x)\leq \mu \text{ for all $x\in \Spec R$}\},\\
\Gr_{G,\mu}&=\{(\cal{E},\beta)\in \Gr_G\mid \mathrm{Inv}(\beta_x)= \mu \text{ for all $x\in \Spec R$}\}.
\end{align*}
Let $N$ be a nonnegative integer.
Define the closed subspace $\ol{\Gr}_{G,N}\subset \Gr_G$ by
\begin{align*}
\ol{\Gr}_{G,N}=\Gr_{G,\leq (N,0,\ldots,0)}.
\end{align*}

There is a fundamental result as follows:
\begin{lemm} \label{genuine}
Let $R$ be a perfect $k$-algebra, and let $\beta\colon \cal{E}_1 \dashrightarrow \cal{E}_2$ be a quasi-isogeny.
For $x\in \Spec R$, write
\[
\mathrm{Inv}(\beta_x)=(m_{x1},\ldots,m_{xn}).
\]
Then $\beta$ is an isogeny if and only if
\[
m_{xn}\geq 0 \text{ for any $x\in \Spec R$.}
\]
Moreover, in this case, $\beta$ induces the inclusion
\[
\beta(\cal{E}_1)\subset \cal{E}_2.
\]
\end{lemm}
\begin{proof}
By the proof of \cite[Lemma 1.5]{Zhumixed}, the quasi-isogeny $\beta$ is a genuine map if and only if $\beta_x$ is a genuine map for any $x\in \Spec R$. 
Moreover, in this case, we have $\beta(\cal{E}_1)\subset \cal{E}_2$.
This shows the lemma.
\end{proof}

By Lemma \ref{genuine}, $\ol{\Gr}_{G,N}$ can be also described as
\[
\ol{\Gr}_{G,N}(R)=\left\{\cal{E}\overset{\beta}{\subset} \cal{E}_0\relmiddle| \begin{array}{l}
\text{$\cal{E}$ is a projective $\cal{O}_R$-submodule of rank $n$,}\\
\text{$\mathrm{Inv}(\beta_x)\leq (N,0,\ldots,0)$ for all $x\in \Spec R$.}
\end{array}\right\}.
\]
\subsection{Isomorphism for $G=\GL_n$}
From this section, we will vary $\cal{O}$, so we write $\cal{O}_1,\cal{O}_2,\ldots$ to distinguish them.
We want to prove
\begin{thm} \label{Grbarisom}
Set $G_1:=\GL_n\otimes_{\bb{Z}} \cal{O}_1$ and $G_2:=\GL_n\otimes_{\bb{Z}} \cal{O}_2$.
If $v_{F_1}(p)\geq N$ and $v_{F_2}(p)\geq N$, then there is a canonical isomorphism
\[
\ol{\Gr}_{G_1,N} \cong \ol{\Gr}_{G_2,N}.
\]
\end{thm}

We only have to consider the case that $\cal{O}_1$ is of mixed characteristic and $\cal{O}_2$ is of equal characteristic.
Namely, we may assume that 
\begin{align}\label{ramiassump}
\begin{cases}
\text{$\cal{O}_1$ is a totally ramified extension of $W(k)$ with $[\cal{O}_1:W(k)]\geq N$,}\\ \text{$\cal{O}_2=k[[t]]$.}
\end{cases}
\end{align}

Then let us denote
\begin{align*}
&\cal{O}:=\cal{O}_1,\\
&\Gr :=\Gr_{G_1}, \Gr^{\flat}:=\Gr_{G_2}
\end{align*}
for simplicity.
Let $\mu\in \bb{X}^+_{\bullet}(T)$ be such that $\mu\leq (N,0,\ldots,0)$.
Then we can write
\[
\mu=\omega_{N_1}+\omega_{N_2}+\cdots +\omega_{N_r},\ (N_1\geq N_2\geq \cdots \geq N_r)
\]
where $\omega_i:=(\overbrace{1,\ldots,1}^{i},0,\ldots,0)\in \ZZ^n$.
Define $\widetilde{\Gr}_{\mu}$ by 
\[
\widetilde{\Gr}_{\mu}(R)=\left\{\cal{E}_r\overset{\beta_r}{\subset}\cdots \overset{\beta_2}{\subset}\cal{E}_1\overset{\beta_1}{\subset} \cal{E}_0\relmiddle| \begin{array}{l}
\text{$\cal{E}_i$'s are projective $W_{\cal{O}}(R)$-modules of rank $n$,}\\
\text{$\beta_i$ is an isogeny of relative position $\omega_{N_i}$.}
\end{array}\right\}.
\]
First, we want to prove the following:
\begin{lemm}\label{Grtildeisom}
Assume (\ref{ramiassump}).
If $\mu \leq (N,0,\ldots,0)=N\omega_1$, then
\[
\widetilde{\Gr}_\mu\cong \widetilde{\Gr}_\mu^\flat.
\]
\end{lemm}
To prove this lemma, we need some preparation.
\begin{lemm}\label{projectivity1}
If $(\cal{E}_{\bullet},\beta_{\bullet})\in \widetilde{\Gr}_{\mu}(R)$, then $\beta_i$ induces a chain of inclusions $\varpi\cal{E}_{i-1}\subset \cal{E}_i\subset \cal{E}_{i-1}$, and $\cal{E}_{i-1}/\cal{E}_i$ is a projective $R$-module of rank $N_i$.

In particular, $\varpi^N\cal{E}_0\subset \cal{E}_i\subset \cal{E}_0$ for all $i$.
\end{lemm}
\begin{proof}
One can use the same argument as in \cite[Lemma 1.5]{Zhumixed}.
\end{proof}
\begin{lemm}\label{projectivity2}
Consider the sequence of $W_{\cal{O}}(R)$-submodules
\[
\varpi^n\cal{E}_0\subset \cal{E}_n\subset \cal{E}_{n-1} \subset \cdots \subset \cal{E}_1 \subset \cal{E}_0.
\]
Assume that $\cal{E}_{i-1}/\cal{E}_i$ is a projective $R$-module annihilated by $\varpi$ for any $i$.
Then $\cal{E}_i$ is a finite projective $W_{\cal{O}}(R)$-module for any $i$.
\end{lemm}
\begin{proof}
Induction on $i$.
If $i=0$, it is clear from the definition.
Assume $i>0$.
Since $\cal{E}_{i-1}/\cal{E}_i$ is $R$-projective, it is a direct summand of $\cal{E}_{i-1}/\varpi\cal{E}_{i-1}$ as an $R$-module, hence also as a $W_{\cal{O}}(R)$-module.

By the  induction hypothesis, $\cal{E}_{i-1}$ is a finite projective $W_{\cal{O}}(R)$-module.
Therefore $\cal{E}_{i-1}/\varpi\cal{E}_{i-1}$ is finitely presented, and so is $\cal{E}_{i-1}/\cal{E}_i$.
Furthermore, it follows that
\begin{align*}
\mathrm{pd}_{W_{\cal{O}}(R)}(\cal{E}_{i-1}/\cal{E}_i)\leq \mathrm{pd}_{W_{\cal{O}}(R)}(\cal{E}_{i-1}/\varpi\cal{E}_{i-1})=1
\end{align*}
where $\mathrm{pd}_{W_{\cal{O}}(R)}(-)$ means its projective dimension over $W_{\cal{O}}(R)$.
This implies that $\cal{E}_i$ is a finite projective $W_{\cal{O}}(R)$-module.
\end{proof}
\begin{lemm} \label{projectivity3}
In the situation of Lemma \ref{projectivity2}, $\mathrm{Inv}(\cal{E}_i\to \cal{E}_{i-1})=\omega_{N_i}$ if and only if the projective module $\cal{E}_{i-1}/\cal{E}_i$ has the constant rank $N_i$.
\end{lemm}
\begin{proof}
Let $\beta$ denote the map $\cal{E}_i\to \cal{E}_{i-1}$.
If $\mathrm{Inv}(\beta)=\omega_{N_i}$, then the projective module $\cal{E}_{i-1}/\cal{E}_i$ has the constant rank $N_i$ by Lemma \ref{projectivity1}.

Conversely, suppose that $\cal{E}_{i-1}/\cal{E}_i$ has the constant rank $N_i$.
As in the proof of \cite[Lemma 1.5]{Zhumixed}, one can show
\begin{align*}
(\cal{E}_{i-1}/\cal{E}_i)\otimes_R k(x)\cong  \mathrm{Coker}(\beta\otimes_{W_{\cal{O}}(R)}W_{\cal{O}}(k(x)))
\end{align*}
for all $x\in \Spec R$.
Since $\cal{E}_{i-1}/\cal{E}_i$ is annihilated by $\varpi$, and $\dim_{k(x)}((\cal{E}_{i-1}/\cal{E}_i)\otimes_R k(x))=N_i$, 
we obtain
\[
\mathrm{Coker}(\beta\otimes_{W_{\cal{O}}(R)}W_{\cal{O}}(k(x)))\cong (W_{\cal{O}}(k(x))/\varpi)^{N_i}
\]
as $W_{\cal{O}}(k(x))$-module.
It means $\mathrm{Inv}(\beta\otimes_{W_{\cal{O}}(R)}W_{\cal{O}}(k(x)))=\omega_{N_i}$ for all $x\in \Spec R$, that is, $\mathrm{Inv}(\beta)=\omega_{N_i}$.
\end{proof}
Now we can prove Lemma \ref{Grtildeisom}.
\begin{proof}[Proof of Lemma \ref{Grtildeisom}]
Define a functor $'\widetilde{\Gr}_{\mu}$ by 
\begin{align*}
'\widetilde{\Gr}_{\mu}(R)\cong \left\{\overline{\cal{E}}_r\subset\cdots \subset \overline{\cal{E}}_1\subset \ol{\cal{E}}_0\relmiddle| \begin{array}{l}
\text{$\overline{\cal{E}}_i$'s are $W_{\cal{O}}(R)/\varpi^N$-submodules of $\ol{\cal{E}}_0$,}\\
\text{$\ol{\cal{E}}_{i-1}/\ol{\cal{E}}_i$ is annihilated by $\varpi$,}\\
\text{and $\ol{\cal{E}}_{i-1}/\ol{\cal{E}}_i$ is a finite projective $R$-module}\\
\end{array}\right\}
\end{align*}
for a perfect $k$-algebra $R$, where $\ol{\cal{E}}_0:=(W_{\cal{O}}/\varpi^N)^n$.
By Lemma \ref{projectivity1}, \ref{projectivity2}, \ref{projectivity3}, we obtain the following bijection:
\begin{align*}
\widetilde{\Gr}_{\mu}(R)&\to {}'\widetilde{\Gr}_{\mu}(R),\\
(\cal{E}_i)_{i=1}^{n}&\mapsto (\cal{E}_i/\varpi^N\cal{E}_0)_{i=1}^{n}.
\end{align*}
This is a natural isomorphism, so $\widetilde{\Gr}_{\mu}\cong {}'\widetilde{\Gr}_{\mu}$ follows.
Similarly, define a functor $'\widetilde{\Gr}^{\flat}_{\mu}$ by 
\begin{align*}
'\widetilde{\Gr}^{\flat}_{\mu}(R)\cong \left\{\overline{\cal{E}}^{\flat}_r\subset\cdots \subset \overline{\cal{E}}^{\flat}_1\subset \ol{\cal{E}}^{\flat}_0\relmiddle| \begin{array}{l}
\text{$\overline{\cal{E}}^{\flat}_i$'s are $R[[t]]/t^N$-submodules of $\ol{\cal{E}}^{\flat}_0$,}\\
\text{$\ol{\cal{E}}^{\flat}_{i-1}/\ol{\cal{E}}^{\flat}_i$ is annihilated by $t$,}\\
\text{and $\ol{\cal{E}}^{\flat}_{i-1}/\ol{\cal{E}}^{\flat}_i$ is a finite projective $R$-module.}\\
\end{array}\right\}
\end{align*}
where $\ol{\cal{E}}^{\flat}_0:=(k[[t]]/t^N)^n$.
Then we have an isomorphism $\widetilde{\Gr}^{\flat}_{\mu}\cong {}'\widetilde{\Gr}^{\flat}_{\mu}$ defined by 
\begin{align*}
\widetilde{\Gr}^{\flat}_{\mu}(R)&\to {}'\widetilde{\Gr}^{\flat}_{\mu}(R),\\
(\cal{E}_i)_{i=1}^{n}&\mapsto (\cal{E}_i/t^N\cal{E}_0)_{i=1}^{n}.
\end{align*}
For $a,b\in R$, the value $[a+b]-[a]-[b]\in pW_{\cal{O}}(R)$ vanishes in $W_{\cal{O}}(R)/\varpi^N$, since $[\cal{O}:W(k)]\geq N$.
Therefore, we obtain a ring isomorphism
\begin{align}\label{mixedeqisom}
W_{\cal{O}}(R)/\varpi^N&\cong R[[t]]/t^N,\\
\sum_{k=0}^{N-1} [a_k]\varpi^k&\mapsto \sum_{k=0}^{N-1} a_kt^k.\notag
\end{align}
Through this isomorphism, we have
\[
'\widetilde{\Gr}_{\mu}\cong {}'\widetilde{\Gr}^{\flat}_{\mu}.
\]
It implies
\[
\widetilde{\Gr}_{\mu} \cong \widetilde{\Gr}^{\flat}_{\mu}.
\]
\end{proof}
\begin{lemm}
There is an isomorphism
\begin{align}\label{Grmuisom}
\Gr_{\mu} \cong \Gr^{\flat}_{\mu}.
\end{align}
\end{lemm}
\begin{proof}
By \cite[Lemma 7.13]{BS}, the natural map
\begin{align*}
\pi\colon \widetilde{\Gr}_{\mu}&\to \Gr_{\leq \mu},\\
(\cal{E}_r\subset \cdots \subset \cal{E}_1\subset \cal{E}_0)&\mapsto (\cal{E}_r\subset \cal{E}_0)
\end{align*}
restricts to an isomorphism
\[
\pi^{-1}\Gr_{\mu}\overset{\sim}{\to} \Gr_{\mu}.
\]
An element $(\cal{E}_{\bullet})\in \widetilde{\Gr}_{\mu}(k)$ is an element of $\pi^{-1}\Gr_{\mu}(k)$ if and only if
\[
\cal{E}_0/\cal{E}_r\cong \bigoplus_{i=1}^n W_{\cal{O}}(k)/\varpi^{\mu_i}
\]
where $\mu=(\mu_1,\ldots,\mu_n)$.
Similarly, the natural map
\begin{align*}
\pi^{\flat}\colon \widetilde{\Gr}^{\flat}_{\mu}&\to \Gr^{\flat}_{\leq \mu},\\
(\cal{E}^{\flat}_r\subset \cdots \subset \cal{E}^{\flat}_1\subset \cal{E}^{\flat}_0)&\mapsto (\cal{E}^{\flat}_r\subset \cal{E}^{\flat}_0)
\end{align*}
restricts to an isomorphism
\[
(\pi^{\flat})^{-1}\Gr^{\flat}_{\mu}\to \Gr^{\flat}_{\mu}.
\]
An element $(\cal{E}^{\flat}_{\bullet})\in \widetilde{\Gr}^{\flat}_{\mu}(k)$ is an element of $\pi^{-1}\Gr^{\flat}_{\mu}(k)$ if and only if
\[
\cal{E}^{\flat}_0/\cal{E}^{\flat}_r\cong \bigoplus_{i=1}^n k[[t]]/t^{\mu_i}.
\]
Therefore, the open subspaces $\pi^{-1}\Gr_{\mu}\subset \widetilde{\Gr}_{\mu}$ and $\pi^{-1}\Gr^{\flat}_{\mu}\subset \widetilde{\Gr}^{\flat}_{\mu}$ correspond to each other by the isomorphism in Lemma \ref{Grtildeisom}.
It implies $\Gr_{\mu} \cong \Gr^{\flat}_{\mu}$.
\end{proof}
If $\beta\colon \cal{E}_2\dashrightarrow \cal{E}_1$ is a quasi-isogeny satisfying $\mathrm{Inv}(\beta_x)\leq N\omega_1$ for all $x\in \Spec R$, then we have 
\begin{align}\label{eqn:E1E2incl}
\varpi^N\cal{E}_1\subset \beta(\cal{E}_2)\subset \cal{E}_1
\end{align}
by applying Lemma \ref{genuine} to $\beta$ and $\dfrac{1}{\varpi^N}\beta^{-1}$.
Hence if $(\cal{E}_1\subset \cal{E}_0)\in \ol{\Gr}_N$, then $\cal{E}_0/\cal{E}_1$ is a $W_{\cal{O}}(R)/\varpi^N$-module.
By the isomorphism in (\ref{mixedeqisom}), $\cal{E}_0/\cal{E}_1$ is also an $R[[t]]/t^N$-module, in particular an $R$-module.
\begin{lemm}\label{projectivity5}
If $(\cal{E}_1\subset \cal{E}_0)\in \ol{\Gr}_N(R)$, then $\cal{E}_0/\cal{E}_1$ is a finite projective $R$-module.
\end{lemm}
\begin{proof}
The proof is almost the same as in \cite[Lemma 1.5]{Zhumixed}.
Namely, for any perfect $R$-algebra $R'$, there is an isomorphism
\begin{align*}
W_{\cal{O},N}(R)\otimes_R R'&\cong R[[t]]/t^N\otimes_R R'\\
&\cong R'[[t]]/t^N\\
&\cong W_{\cal{O},N}(R').
\end{align*}
Therefore, for $x\in \Spec R$, we obtain
\begin{align*}
(\cal{E}_0/\cal{E}_1)\otimes_R k(x)&\cong (\cal{E}_0/\cal{E}_1)\otimes_{W_{\cal{O},N}(R)}W_{\cal{O},N}(k(x))\\
&\cong (\cal{E}_0/\cal{E}_1)\otimes_{W_{\cal{O}}(R)}W_{\cal{O}}(k(x))\\
&\cong (\cal{E}_0\otimes_{W_{\cal{O}}(R)}W_{\cal{O}}(k(x)))/(\cal{E}_1\otimes_{W_{\cal{O}}(R)}W_{\cal{O}}(k(x)))\\
&=\mathrm{Coker}(\beta_x).
\end{align*}
Since $\mathrm{Inv}(\beta_x)\leq N\omega_1$, we have $\dim \mathrm{Coker}(\beta_x)=N$ for all $x\in \Spec R$.
Hence $\dim (\cal{E}_0/\cal{E}_1)\otimes_R k(x)$ is constant on $\Spec R$.

On the other hand, $\cal{E}_0/\cal{E}_1$ is the cokernel of $\cal{E}_1/\varpi^N \to \cal{E}_0/\varpi^N$.
Also, $\cal{E}_i/\varpi^N\ (i=0,1)$ is a finite projective $W_{\cal{O}}(R)$-module, and hence a finite projective $R$-module.
Therefore, $\cal{E}_0/\cal{E}_1$ is finitely presented as an $R$-module.

Over a reduced ring, a finitely presented module whose fiber dimensions are constant is locally free (see \cite[Tag0FWG]{St}).
The lemma follows.
\end{proof}
Fix an isomorphism $W_{\cal{O},N}(k)^n \cong k^{nN}$ of vector spaces.
By Lemma \ref{projectivity5}, we obtain a morphism $i_{\ol{\Gr}_{N}}\colon \ol{\Gr}_{N}\to \Gr(nN)^{p^{-\infty}}$ defined by 
\begin{align} \label{embusualgr}
\ol{\Gr}_{N}(R)\ni (\cal{E}_1\subset \cal{E}_0) \mapsto (\cal{E}_0/\cal{E}_1)\in \Gr(nN)^{p^{-\infty}}(R)
\end{align}
where $\Gr(nN)$ is a usual Grassmannian, classifying finite dimensional subspaces in $k^{nN}$.
\begin{lemm}\label{embusualgrclosed}
The morphism $i_{\ol{\Gr}_{N}}$ is a closed immersion.
\end{lemm}
\begin{proof}
We know that $\ol{\Gr}_{N}$ is perfectly proper (i.e. separated and universally closed) over $k$ and $\Gr(nN)^{p^{-\infty}}$ is separated (see \cite[Lemma 3.4]{BS}).
Therefore $i_{\ol{\Gr}_{N}}$ is perfectly proper.

Furthermore, the map between the sets of $R$-valued points 
\[
i_{\ol{\Gr}_{N}}(R)\colon \ol{\Gr}_{N}(R)\to \Gr(nN)^{p^{-\infty}}(R)
\]
is injective for any perfect $k$-algebra $R$.
In particular, $i_{\ol{\Gr}_{N}}(K)$ is injective for any algebraically closed field $K$.
This implies that $i_{\ol{\Gr}_{N}}$ is universally injective.

By \cite[Lemma 3.8]{BS}, a universal homeomorphism between perfect schemes is an isomorphism.
It follows that a perfectly proper and universally injective morphism between perfect schemes is a closed immersion.
This proves the claim.
\end{proof}
Now we can prove Theorem \ref{Grbarisom}.
\begin{proof}[Proof of Theorem \ref{Grbarisom}]
From Lemma \ref{embusualgrclosed}, we obtain a closed immersion $i_{\ol{\Gr}_{N}}\colon \ol{\Gr}_{N}\hookrightarrow \Gr(nN)^{p^{-\infty}}$ by fixing an isomorphism 
\begin{align}\label{vecspisom}
W_{\cal{O},N}(k)^n \cong k^{nN}
\end{align}
of vector spaces.
Consider the isomorphism
\begin{align*}
(k[[t]]/t^{N})^n \overset{(\ref{mixedeqisom})}{\cong} W_{\cal{O},N}(k)^n \overset{(\ref{vecspisom})}{\cong} k^{nN}.
\end{align*}
Then similarly we obtain a closed immersion
\[
i^{\flat}_{\ol{\Gr}_{N}}\colon \ol{\Gr}_{N}^{\flat} \hookrightarrow \Gr(nN)^{p^{-\infty}}.
\]
Let $\mu\in \mathbb{X}^{+}_{\bullet}(T)$ be such that $\mu\leq N\omega_1$.
Then by construction, the following diagram is commutative:
\[
\vcenter{
\xymatrix{
\Gr_{\mu}\ar@{^{(}->}[r]\ar[d]^{\rotatebox{90}{$\sim$}}_{(\ref{Grmuisom})} &\ol{\Gr}_N \ar@{^{(}->}[r]& \Gr(nN)^{p^{-\infty}}\ar@{=}[d]\\
\Gr^{\flat}_{\mu}\ar@{^{(}->}[r] &\ol{\Gr}^{\flat}_N \ar@{^{(}->}[r]& \Gr(nN)^{p^{-\infty}}.
}
}
\]
Since $\ol{\Gr}_N=\ds\bigcup_{\mu\leq N\omega_1} \Gr_{\mu}$ and $\ol{\Gr}^{\flat}_N=\ds\bigcup_{\mu\leq N\omega_1} \Gr^{\flat}_{\mu}$, it follows that $\ol{\Gr}_N$ and $\ol{\Gr}^{\flat}_N$ coincide as perfect closed subschemes of $\Gr(nN)^{p^{-\infty}}$.
\end{proof}
More precisely, the isomorphism $\ol{\Gr}_{N}\cong \ol{\Gr}^{\flat}_N$ has some equivariance.
To explain this, consider the following lemma:
\begin{lemm}\label{lemm:Grbartriv}
The $L^+\GL_n^{(N)}$-action on $\ol{\Gr}_{N}$ is trivial.\\
Similarly, $L^+\GL_n^{\flat, (N)}$-action on $\ol{\Gr}^{\flat}_{N}$ is trivial.
\end{lemm}
Here, $L^+\GL_n^{(N)},L^+\GL_n^{\flat, (N)}$ are the $N$-th congruence groups.
\begin{proof}
By the $L^+\GL_n$-action on $\ol{\Gr}_{N}$, an element $A\in L^+\GL_n(R)$ sends $(\cal{E}_1\overset{\beta}{\subset} \cal{E}_0)\in \ol{\Gr}_{N}(R)$ to $(\cal{E}_1\overset{A \circ \beta}{\subset} \cal{E}_0)\in \ol{\Gr}_{N}(R)$.

By Lemma \ref{embusualgrclosed} and (\ref{eqn:E1E2incl}), a point $(\cal{E}_1\subset \cal{E}_0)\in \ol{\Gr}_{N}(R)$ is completely determined by its quotient $\cal{E}_1/\varpi^N\cal{E}_0$.

Since $\cal{E}_1/\varpi^N\cal{E}_0$ does not change by the $L^+\GL_n^{(N)}$-action, the lemma follows.
\end{proof}
By Lemma \ref{lemm:Grbartriv}, $L^N\GL_n$ acts on $\ol{\Gr}_{N}$.
Similarly, $L^N\GL_n^{\flat}$ acts on $\ol{\Gr}^{\flat}_N$.
But there is an isomorphism
\begin{align}\label{eqn:mixedeqisomL}
L^N\GL_n\cong L^N\GL_n^{\flat}
\end{align}
by the isomorphism (\ref{mixedeqisom}).
\begin{prop}\label{prop:Grbarequiv}
The isomorphism $\ol{\Gr}_{N}\cong \ol{\Gr}^{\flat}_N$ in Theorem \ref{Grbarisom} is $L^N\GL_n$-equivariant through the isomorphism (\ref{eqn:mixedeqisomL}).
\end{prop}
\begin{proof}
This follows from the construction of the isomorphism $\ol{\Gr}_{N}\cong \ol{\Gr}^{\flat}_N$.
\end{proof}
Write $\ol{\GL}_n$ for the perfection of the general linear group over $k$.
Then by Proposition \ref{prop:Grbarequiv}, the isomorphism $\ol{\Gr}_{N}\cong \ol{\Gr}^{\flat}_N$ is in particular $\ol{\GL}_n$-equivariant through the map $\ol{\GL}_n\to L^N\GL_n$ coming from the natural map $k\to k[[t]]/t^N$.
\subsection{Isomorphism for general $G$}
If $N$ is a positive integer, and if $v_{F_1}(p)\geq N$ and $v_{F_2}(p)\geq N$ hold, then by the same argument as (\ref{mixedeqisom}), there is a ring isomorphism
\begin{align}\label{eqn:mixedmixedisom}
\cal{O}_1/\varpi_1^N&\cong \cal{O}_2/\varpi_2^N,\\
\sum_{i=0}^{N-1}[a_i]\varpi_1^i&\mapsto \sum_{i=0}^{N-1}[a_i]\varpi_2^i
\end{align}
where in the equal characteristic case, we define $[\cdot]$ by $[a]:=a$ for any $a\in k$.
\begin{thm}[restatement of Theorem \ref{thm:Grisomgeneralintro}] \label{thm:Grisomgeneral}
Let $(G_{\ZZ},B_{\ZZ},T_{\ZZ})$ be a triple of a split reductive group over $\ZZ$, its Borel group and its maximal torus such that $T\subset B$.
Put $(G_1,B_1,T_1)=(G_{\ZZ},B_{\ZZ},T_{\ZZ})\otimes_{\ZZ}\cal{O}_1$ and $(G_2,B_2,T_2)=(G_{\ZZ},B_{\ZZ},T_{\ZZ})\otimes_{\ZZ}\cal{O}_2$.
Let $\mu$ be an element in $\mathbb{X}_{\bullet}^+(T_{\ZZ})$.
Then there exists a constant $N:=N_{G_{\ZZ},\mu}\in \mathbb{Z}_{>0}$ depending only on the isomorphism class of $G_{\ZZ}$ and $\mu$ such that if $v_{F_1}(p)\geq N$ and $v_{F_2}(p)\geq N$, then there is a ``canonical'' isomorphism
\begin{align*}
\alpha_{G_{\ZZ}, \mu,\cal{O}_1,\cal{O}_2}\colon \Gr_{G_1,\leq \mu} \cong \Gr_{G_2,\leq \mu}.
\end{align*}
Moreover, the Schubert cells $\Gr_{G_1,\lambda}$ and $\Gr_{G_2,\lambda}$ correspond under $\alpha_{G_{\ZZ}, \mu,\cal{O}_1,\cal{O}_2}$ for any $\lambda\leq \mu$.
\end{thm}
Here the term ``canonical'' means the following:
\begin{enumerate}
\item \label{item:canonicity:trans} If $v_{F_3}(p)\geq N$, then
\[
\alpha_{G_{\ZZ}, \mu,\cal{O}_1,\cal{O}_3}=\alpha_{G_{\ZZ}, \mu,\cal{O}_2,\cal{O}_3}\circ \alpha_{G_{\ZZ}, \mu,\cal{O}_1,\cal{O}_2}.
\]
\item \label{item:canonicity:functorial} Let $(G'_{\bb{Z}},B'_{\ZZ},T'_{\ZZ})$ be another triple of a reductive group over $\ZZ$.
Put $(G'_1,B'_1,T'_1)=(G'_{\bb{Z}},B'_{\ZZ},T'_{\ZZ})\otimes_{\ZZ} \cal{O}_1$ and $(G'_2,B'_2,T'_2)=(G'_{\bb{Z}},B'_{\ZZ},T'_{\ZZ})\otimes_{\ZZ} \cal{O}_2$.
Let $f_{\ZZ}\colon G_{\ZZ}\to G'_{\ZZ}$ be a homomorphism of algebraic groups over $\ZZ$ such that $f_{\ZZ}(B_{\ZZ})\subset B'_{\ZZ}$ and $f_{\ZZ}(T_{\ZZ})\subset T'_{\ZZ}$ hold.
Let $\mu\in \bb{X}_{\bullet}^+(T_{\ZZ})$.
We denote by $f(\mu)$ the image of $\mu$ under the map $\bb{X}^+_{\bullet}(T_{\ZZ})\to \bb{X}^+_{\bullet}(T'_{\ZZ})$ induced by $f_{\ZZ}$.
Take $\mu'\in \bb{X}^+_{\bullet}(T'_{\ZZ})$ such that $\mu'\geq f(\mu)$.
Put $\wt{N}:=\max\{N_{G_{\ZZ},\mu},N_{G'_\ZZ,\mu'}\}$.
If $v_{F_1}(p)\geq \wt{N}$ and $v_{F_2}(p)\geq \wt{N}$ hold, then the following square is commutative:
\[
\xymatrix@C=70pt{
\Gr_{G_1,\leq \mu}\ar[r]_-{\sim}^-{\alpha_{G_{\ZZ}, \mu,\cal{O}_1,\cal{O}_2}}\ar[d]_{f_1}&\Gr_{G_2,\leq \mu}\ar[d]^{f_2}\\
\Gr_{G'_1,\leq \mu'}\ar[r]^-{\sim}_-{\alpha_{G'_{\ZZ}, \mu',\cal{O}_1,\cal{O}_2}}&\Gr_{G'_2,\leq \mu'},
}
\]
where the vertical maps are the maps induced by $f_{\ZZ}$.
In particular, if $\lambda$ and $\mu \in \bb{X}_{\bullet}^+(T_{\ZZ})$ satisfy $\lambda \leq \mu$, then
\[
\alpha_{G_{\ZZ}, \lambda,\cal{O}_1,\cal{O}_2}=(\alpha_{G_{\ZZ}, \mu,\cal{O}_1,\cal{O}_2})|_{\Gr_{G_1,\leq \lambda}}
\]
\end{enumerate}
\begin{proof}[Proof of Theorem \ref{thm:Grisomgeneral}]
It is enough to construct an isomorphism $\alpha_{G_{\ZZ}, \mu,\cal{O}_1,\cal{O}_2}$ satisfying the canonicity condition \ref{item:canonicity:functorial} in the case where $\cal{O}_1$ is of mixed characteristic and $\cal{O}_2$ is of equal characteristic.
Put
\begin{align*}
&F:=F_1,\\
&G:=G_1,\ G^{\flat}:=G_2,\\
&\Gr_{G}:=\Gr_{G_1},\ \Gr^{\flat}_G:=\Gr_{G_2}
\end{align*}
and so on.

Choose a closed embedding $i_{\ZZ}\colon G_{\ZZ}\hookrightarrow \GL_{n,\ZZ}$ of algebraic groups for some $n$.
Then the induced maps
\begin{align*}
i&\colon \Gr_G\hookrightarrow \Gr_{\GL_n},\\
i^{\flat}&\colon \Gr^{\flat}_G\hookrightarrow \Gr^{\flat}_{\GL_n}
\end{align*}
are closed immersions (see \cite[Proposition 1.20]{Zhumixed} and \cite[Corollary
9.7.7]{Ap}).
Since $\Gr_{G,\leq \mu}$ is a connected scheme, we obtain
\begin{align*}
i(\Gr_{G,\leq \mu})\subset \varpi^{m}\ol{\Gr}_{N}
\end{align*}
for some $m$ and $N$.
Choose $i_{\ZZ}$ and $N$ so that $N$ is the smallest, and define $N_{G_{\ZZ},\mu}$ as this minimum $N$.
Assume $v_{F}(p)\geq N_{G_{\ZZ},\mu}$.
Then by Theorem \ref{Grbarisom}, there is a canonical isomorphism
\begin{align}\label{eqn:pimGrbarisom}
\varpi^{m}\ol{\Gr}_{N} \cong t^{m}\ol{\Gr}^{\flat}_{N}.
\end{align}
Furthermore, this isomorphism is $L^N\GL_n(\cong L^N\GL_n^{\flat})$-equivariant by Proposition \ref{prop:Grbarequiv}.
The image $i^{\flat}(\Gr^{\flat}_{G,\leq \mu})$ is the smallest $L^+G^{\flat}$-stable subspace of $\Gr^{\flat}_{\GL_n}$ containing $i^{\flat}(t^{\lambda})$ for all $\lambda\leq \mu$.
Since $i(\varpi^{\lambda})=\varpi^{i(\lambda)}$ and $i^{\flat}(t^{\lambda})=t^{i(\lambda)}$ correspond under the isomorphism (\ref{eqn:pimGrbarisom}), it follows that $i^{\flat}(t^{\lambda})\in t^{m}\ol{\Gr}^{\flat}_{N}$ for all $\lambda\leq \mu$, and that 
\[
i^{\flat}(\Gr^{\flat}_{G,\leq \mu})\subset t^{m}\ol{\Gr}^{\flat}_{N}.
\]
Therefore, $i^{\flat}(\Gr^{\flat}_{G,\leq \mu})$ is the smallest $L^+G^{\flat}$-stable subspace of $t^{m}\ol{\Gr}^{\flat}_{N}$ and $i(\Gr_{G,\leq \mu})$ is the smallest $L^+G$-stable subspace of $t^{m}\ol{\Gr}^{\flat}$.
The subgroup $L^+G^{(N)}(\subset L^+\GL_n^{(N)})$ acts trivially on $\varpi^{m}\ol{\Gr}_{N}$ by Lemma \ref{lemm:Grbartriv}.
Similarly, $i^{\flat}(\Gr^{\flat}_{G,\lambda})$ is the $L^NG^{\flat}$-orbit of $i^{\flat}(t^{\lambda})$.
Hence, $i^{\flat}(\Gr^{\flat}_{G,\leq \mu})$ is the smallest $L^NG^{\flat}$-stable subspace of $t^{m}\ol{\Gr}^{\flat}_{N}$ and $i(\Gr_{G,\leq \mu})$ is the smallest $L^NG$-stable subspace of $t^{m}\ol{\Gr}^{\flat}$.

The isomorphism (\ref{mixedeqisom}) induces isomorphisms $(G\ \rm{mod}\ \varpi^N)\overset{\sim}{\to} (G^{\flat}\ \rm{mod}\ t^N)$ and $(\GL_n\ \rm{mod}\ \varpi^N)\overset{\sim}{\to} (\GL_n^{\flat}\ \rm{mod}\ t^N)$.
Moreover, the embedding $i_{\ZZ}$ induces embeddings $(G\ \rm{mod}\ \varpi^N)\hookrightarrow (\GL_n\ \rm{mod}\ \varpi^N)$ and $(G^{\flat}\ \rm{mod}\ t^N)\hookrightarrow (\GL^{\flat}_n\ \rm{mod}\ t^N)$.
By construction, the square
\[
\xymatrix{
(G\ \rm{mod}\ \varpi^N)\ar@{^{(}->}[r]\ar[d]_{\phi}^{\rotatebox{90}{$\sim$}}&(\GL_n\ \rm{mod}\ \varpi^N)\ar[d]_{\psi}^{\rotatebox{90}{$\sim$}}\\
(G^{\flat}\ \rm{mod}\ t^N)\ar@{^{(}->}[r]&(\GL_n^{\flat}\ \rm{mod}\ t^N)
}
\]
is commutative.
Therefore, there is a commutative square of the form
\[
\xymatrix{
L^NG\ar@{^{(}->}[r]\ar[d]^{\rotatebox{90}{$\sim$}}&L^N\GL_n\ar[d]^{\rotatebox{90}{$\sim$}}\\
L^NG^{\flat}\ar@{^{(}->}[r]&L^N\GL_n^{\flat}.
}
\]
Since (\ref{eqn:pimGrbarisom}) is $L^N\GL_n(\cong L^N\GL_n^{\flat})$-equivariant, it is also $L^NG(\cong L^NG^{\flat})$-equivariant.
Hence we obtain the isomorphism
\[
\alpha_{G_{\ZZ},\mu,\cal{O}_1,k[[t]]}\colon \Gr_{G,\leq \mu}\cong \Gr^{\flat}_{G,\leq \mu}.
\]

Finally, we want to prove the canonicity condition \ref{item:canonicity:functorial}.
(This implies that the isomorphism $\alpha_{G_{\ZZ},\mu,\cal{O}_1,k[[t]]}$ does not depend on the choice of $\wt{i}$, as we will mention later.)
Let $(G'_{\ZZ},B'_{\ZZ},T'_{\ZZ})$ be another triple.
Let $\mu'$ be an element in $\mathbb{X}_{\bullet}^+(T_{\ZZ})$.
Let $f_{\ZZ}\colon G_{\ZZ}\to G'_{\ZZ}$ be a homomorphism of algebraic groups over $\ZZ$ such that $f_{\ZZ}(B_{\ZZ})\subset B'_{\ZZ}$ and $f_{\ZZ}(T_{\ZZ})\subset T'_{\ZZ}$ hold.
Choose embeddings $i_{\ZZ}\colon G_{\ZZ}\hookrightarrow \GL_{n,\ZZ}$ and $i'_{\ZZ}\colon G'_{\ZZ}\hookrightarrow \GL_{n',\ZZ}$ so that these embeddings induce closed immersions $i\colon \Gr_{G,\mu}\hookrightarrow \varpi^m\Gr_{\GL_n,\leq N_{G_{\ZZ},\mu}\omega_1}$ and $i\colon \Gr_{G',\mu'}\hookrightarrow \varpi^{m'}\Gr_{\GL_{n'},\leq N_{G'_{\ZZ},\mu'}\omega_1}$ for some $m,m'$.
Put $N:=N_{G_{\ZZ},\mu}, N':=N_{G'_{\ZZ},\mu'}$ and $\wt{N}:=\max\{N,N'\}$.
Consider two isomorphisms

The closed immersion
\[
(i_{\ZZ},i'_{\ZZ}\circ f_{\ZZ})\colon G\hookrightarrow GL_n\times GL_{n'}
\]
induces maps
\begin{align*}
\Gr_{G,\leq \mu}&\to \Gr_{\GL_n\times \GL_{n'}}\cong \Gr_{\GL_n}\times \Gr_{\GL_{n'}}\\
\Gr^{\flat}_{G,\leq \mu}&\to \Gr^{\flat}_{\GL_n\times \GL_{n'}}\cong \Gr^{\flat}_{\GL_n}\times \Gr^{\flat}_{\GL_{n'}}.
\end{align*}
By the same argument as above, the image of the latter is contained in $(t^m\Gr^{\flat}_{\GL_n,\leq N\omega_1})\times (t^{m'}\Gr^{\flat}_{\GL_{n'},\leq N'\omega_1})$ and there is an isomorphism which makes the following diagram commute, assuming only that $v_F(p)\geq \wt{N}$:
\[
\vcenter{
\xymatrix@!C{
\Gr_{G,\leq \mu}\ar[r]\ar@{.>}[d]&(\varpi^m\Gr_{\GL_n,\leq N\omega_1})\times (\varpi^{m'}\Gr_{\GL_{n'},\leq N'\omega_1})\ar[d]\\
\Gr^{\flat}_{G,\leq \mu}\ar[r]&(t^m\Gr^{\flat}_{\GL_n,\leq N\omega_1})\times (t^{m'}\Gr^{\flat}_{\GL_{n'},\leq N'\omega_1}).
}
}
\]

Then we obtain the diagram
\[
\xymatrix{
\Gr_{G,\leq \mu}\ar[rr]\ar[dd]^{\rotatebox{90}{$\sim$}}\ar[rd]&&\Gr_{G',\leq \mu'}\ar[dd]^{\rotatebox{90}{$\sim$}}\ar[rd]&\\
&(\varpi^m\Gr_{\GL_n,\leq N\omega_1})\times (\varpi^{m'}\Gr_{\GL_{n'},\leq N'\omega_1})\ar[rr]^{\rm{pr}_2}\ar[dd]^{\rotatebox{90}{$\sim$}}&&\varpi^{m'}\Gr_{\GL_{n'},\leq N'\omega_1}\ar[dd]^{\rotatebox{90}{$\sim$}}\\
\Gr^{\flat}_{G,\leq \mu}\ar[rd]&&\Gr^{\flat}_{G',\leq \mu'}\ar[rd]&\\
&(t^m\Gr^{\flat}_{\GL_n,\leq N\omega_1})\times (t^{m'}\Gr^{\flat}_{\GL_{n'},\leq N'\omega_1})\ar[rr]^{\rm{pr}_2}&&t^{m'}\Gr^{\flat}_{\GL_{n'},\leq N'\omega_1}
}
\]
where all the squares are commutative.

We have the map $\Gr^{\flat}_{G,\leq \mu}\to \Gr^{\flat}_{G',\leq \mu'}$ which makes the cube commutative.
This map coincides with the map induced by $f_{\ZZ}$.
This proves the canonicity.
The canonicity shows that $\alpha_{G_{\ZZ},\mu,\cal{O}_1,k[[t]]}$ does not depend on the choice of $i_{\ZZ}$, by considering the case where $G_{\ZZ}=G'_{\ZZ}$, $f_{\ZZ}=\rm{id}_{G_{\ZZ}}$ in the above argument.
\end{proof}
\section{Geometric Satake equivalence and Springer correspondence}\label{scn:GSSC}
\subsection{Preliminaries on Springer correspondence}
The results in this subsection are parallel to \cite[\S 2.7]{AHR}.
Let $\fk{g}_{\rs}\subset \fk{g}$ be the subset consisting of regular semisimple elements.
Recall that there is a diagram
\[
\xymatrix{
\bar{G}\times^{\bar{B}} \fk{u}
\ar@{^{(}->}[r]\ar[d]_{\mu_\cal{N}}
\ar@{}[rd]|{\square}
&\bar{G}\times^{\bar{B}} \fk{b}\ar[d]^{\mu_{\fk{g}}}\ar@{}[rd]|{\square}&\bar{G} \times^{\bar{B}}(\fk{g}_\rs\cap \fk{b}) \ar@{_{(}->}[l]\ar[d]^{\mu_{\fk{g}}^\rs}\\
\cal{N}_G\ar@{^{(}->}_{i_{\fk{g}}}[r]&\fk{g}&\fk{g}_\rs \ar@{_{(}->}[l]^-{j_{\fk{g}}}
}
\]
where all horizontal maps are inclusions, and all vertical maps send $(g,x)$ to $g\cdot x$.
Since $\mu_{\fk{g}}$ is proper and small,
\begin{align*}
\Groth:=(\mu_{\fk{g}})_!{\Qlb}_{\bar{G}\times^{\bar{B}} \fk{b}}[\dim \fk{g}]
\end{align*}
is a $\bar{G}$-equivariant perverse sheaf on $\fk{g}$.
There is a canonical isomorphism
\begin{align*}
\Groth\cong (j_{\fk{g}})_{!*}((\mu_{\fk{g}}^\rs)_!\Qlb[\dim\fk{g}])
\end{align*}
 and $\mu^\rs_{\fk{g}}$ is a Galois covering with Galois group $W_G$.
Hence we obtain a $W_G$-action on $\Groth$.

Moreover, since $\mu_{\cal{N}}$ is proper and semismall,
\[
\Spr:=(\mu_{\cal{N}})_!{\Qlb}_{\bar{G}\times^{\bar{B}}\fk{u}}[\dim\cal{N}_G]
\]
is a $\bar{G}$-equivariant perverse sheaf on $\cal{N}_G$.
There is a canonical isomorphism
\[
\Spr\cong (i_{\fk{g}})^*\Groth[-r]
\]
 where $r=\rm{rk}(\bar{G})=\dim\fk{g}-\dim\cal{N}_G$.
Hence we obtain a $W_G$-action on $\Spr$.
This induces a functor
\begin{align*}
\bb{S}_G\colon \Perv_{\bar{G}}(\cal{N}_G^{p^{-\infty}},\Qlb)\cong \Perv_{\bar{G}}(\cal{N}_G,\Qlb)&\to \Rep(W_G,\Qlb),\\
M&\mapsto \rm{Hom}_{\Perv_{\bar{G}}(\cal{N}_G,\Qlb)}(\Spr,M).
\end{align*}
For the first equivalence, note that the \'{e}tale topos on a scheme is equivalent to that on its perfection.
\subsection{Functor $\scr{S}_G^{\sm}$}
The results in this subsection are parallel to \cite[\S 2.3]{AHR}.
Let $z_G\colon \Gr_G^{\circ}\hookrightarrow \Gr_G$ be the inclusion where $\Gr_G^{\circ}$ is the connected component of $\Gr$ containing $\varpi^0$.
The functor
\[
(z_G)_*\colon \Perv_{L^+G}(\Gr_G^{\circ},\Qlb)\to \Perv_{L^+G}(\Gr_G,\Qlb)
\]
is fully faithful.
The essential image of $\scr{S}_G\circ (z_G)_*$ is the subcategory $\Rep(\check{G},\Qlb)^{Z(\check{G})}$ of $\Rep(\check{G},\Qlb)$ consisting of representations on which $Z(\check{G})$ acts trivially.
Let $\mathbf{I}_{\check{G}}\colon \Rep(\check{G},\Qlb)^{Z(\check{G})}\to \Rep(\check{G},\Qlb)$ be the inclusion.
There is a unique equivalence of categories
\[
\scr{S}_G^\circ \colon \Perv_{L^+G}(\Gr_G^{\circ},\Qlb)\to \Rep(\check{G},\Qlb)^{Z(\check{G})}
\]
such that 
\[
\mathbf{I}_{\check{G}}\circ \scr{S}_{G}^\circ=\scr{S}_G\circ (z_G)_*.
\]
Since $(z_G)_*$ is left adjoint to $(z_G)^!$ and $\mathbf{I}_{\check{G}}$ is left adjoint to $(-)^{Z(\check{G})}\colon \Rep(\check{G},\Qlb)\to \Rep(\check{G},\Qlb)^{Z(\check{G})}$, there is a canonical isomorphism
\begin{align}\label{eqn:ZGinvSG}
(-)^{Z(\check{G})}\circ \scr{S}_{G}\Iff \scr{S}_G^\circ\circ (z_G)^!.
\end{align}

Let $f_G\colon \Gr_G^{\sm}\hookrightarrow \Gr_G$ be the inclusion.
The functor
\[
(f_G)_*\colon \Perv_{L^+G}(\Gr_G^{\sm},\Qlb)\to \Perv_{L^+G}(\Gr_G,\Qlb)
\]
is fully faithful.
The essential image of $\scr{S}_G\circ (f_G)_*$ is the subcategory $\Rep(\check{G},\Qlb)_{\sm}$ of $\Rep(\check{G},\Qlb)$ consisting of representations whose $\check{T}$-weights are small.
Let $\bb{I}_{\check{G}}\colon \Rep(\check{G},\Qlb)_{\sm}\to \Rep(\check{G},\Qlb)$ be the inclusion.
There is a unique equivalence of categories
\[
\scr{S}_G^\sm\colon \Perv_{L^+G}(\Gr_G^{\sm},\Qlb)\to \Rep(\check{G},\Qlb)_{\sm}
\]
such that 
\[
\bb{I}_{\check{G}}\circ \scr{S}_{G}^\sm=\scr{S}_G\circ (f_G)_*.
\]

Let $f_G^{\circ}\colon \Gr^{\sm}\hookrightarrow \Gr^{\circ}$ and $\bb{I}^{\circ}_{\check{G}}\colon \Rep(\check{G},\Qlb)_{\sm}\to \Rep(\check{G},\Qlb)^{Z(\check{G})}$ be the inclusions.
Since $f_G=z_G\circ f_G^{\circ}$ and $\bb{I}_{\check{G}}=\mathbf{I}_{\check{G}}\circ \bb{I}^{\circ}_{\check{G}}$ hold, we obtain a canonical isomorphism
\begin{align}\label{eqn:IcircSGsm}
\bb{I}^{\circ}_{\check{G}}\circ \scr{S}_G^{\sm}\Iff \scr{S}_G^{\circ}\circ (f_G^{\circ})_*.
\end{align}
\subsection{Functor $\Phi_{\check{G}}$}
 For $V\in \Rep(\check{G},\Qlb)$, the Weyl group $W_G=N_{\check{G}}(\check{T})/\check{T}$ naturally acts on the zero weight space $V^{\check{T}}$.
Let $\varepsilon_{W_G}$ be the sign character of the Coxeter group $W_G$.
Then we obtain the functor 
\begin{align*}
\Phi_{\check{G}}\colon \Rep(\check{G},\Qlb)_{\sm}&\to \Rep(W_G,\Qlb),\\
V&\mapsto V^{\check{T}}\otimes \varepsilon_{W_G}.
\end{align*}
\subsection{Functor $\Psi_{G}$}\label{ssc:Psiexplain}
Let $v_F$ be the normalized valuation on $F$.
The functor $\Psi_{G}$ will not be defined unless $p$ is good for $G$ and $v_F(p)\geq C$ where $C=C_{\bar{G}}\in \bb{N}$ is a constant depending only on the isomorphism class of $\bar{G}$.

We first show that if $p$ is good for $G$ and $v_F(p)\geq C$, there are morphisms $j_G\colon \cal{M}_G\hookrightarrow \Gr^{\sm}_G$ and $\pi_G\colon \cal{M}_G\to \fk{g}$, where $j_G$ is an open immersion and $\pi_G$ is the perfection of a finite morphism.
These induce an exact functor
\begin{align} \label{eqn:defPhi}
\Psi_G:=(\pi_G)_*\circ (j_G)^!\colon \Perv_{L^+G}(\Gr^{\sm},\Qlb)\to \Perv_{\bar{G}}(\cal{N}_G^{p^{-\infty}},\Qlb).
\end{align}

If $\cal{O}=k[[t]]$ and $G=\bar{G}\otimes_k k[[t]]$, then we can define a non-perfect version of $\cal{M}_G$ by
\[
\cal{M}'_G:=\Gr^{\prime,\sm}_G\cap \Gr^-_{G,0}
\]
where $\Gr^-_{G,0}:=\bar{G}(k[t^{-1}])\cdot t^{0}\subset \Gr'_G$.
Put $\fk{G}=\rm{Ker}(\bar{G}(k[t^{-1}])\to \bar{G}(k))$. Then there is a homomorphism 
\begin{align}\label{eqn:isomfkGGr-}
\fk{G}\xrightarrow{\sim} \Gr^-_{G,0},\ g\mapsto g\cdot t^0.
\end{align}
On the other hand, the Lie algebra $\fk{g}$ is a kernel of the map $\bar{G}(k[t^{-1}]/t^{-2})\to \bar{G}(k)$.
Hence there is a canonical map
\[
\pi_G^{\dagger} \colon \Gr^-_{G,0}\to \fk{g}.
\]
\begin{prop}
In the above situation (i.e. $\cal{O}=k[[t]]$ and $p$ is good), we have $\pi_G^{\prime,\dagger}(\cal{M}'_G)\subset \cal{N}_G$.
Moreover, $\pi'_G:=\pi_G^{\prime,\dagger}|_{\cal{M}'_G}\colon \cal{M}'_G\to \cal{N}_G$ is a finite morphism.
\end{prop}
\begin{proof}
By the paragraph after \cite[Lemma 2.4]{AH}, we can assume that $G$ is a simply-connected and simple group.
Then, this can be proved by the same argument as \cite[Theorem 1.1]{AH}, using the assumption that $p$ is good.

The reason why the same argument as in \cite{AH} applies when $p$ is good is as follows:
\begin{enumerate}
\item The case where $\bar{G}$ is a classical group\\
In this case, there is a `standard' representation $\bar{G}\to GL_n$.
Let $\bar{G}'$ be its image. (Thus, $\bar{G}'=\bar{G}$ if $\bar{G}=SL_n$ or $Sp_n$, and $\bar{G}'=SO_n$ if $\bar{G}=Spin_n$.)
\begin{enumerate}
\item We standardly regard $\bb{X}_{\bullet}(T)$ as a subset of $\bb{Z}^n$, and write the subset of small cocharacters explicitly.
\item The $\bar{G}$-orbits of $\cal{N}_G$ are standardly parametrized by a certain set of partitions.
\item By (\ref{eqn:isomfkGGr-}) and since an isogeny of groups induces an isomorphism of $\fk{G}$, we have
\[
\Gr_{G,0}^-\cong \fk{G}\cong \fk{G}':=\rm{Ker}(\bar{G}'(k[t^{-1}])\to \bar{G}'(k))\subset GL_n(k[t^{-1}]).
\]
Therefore, we can identify $\Gr_{G,0}^-$ with a subset of $GL_n(k[t^{-1}])$.
Then, the map $\pi^{\dagger}_{G}\colon \Gr_{G,0}^-\to \fk{g}\subset \fk{gl}_n(k)$ is the restriction of
\[
GL_n(k[t^{-1}]) \ni x_0+x_1t^{-1}+x_2t^{-2}+\cdots \mapsto x_{1}\in \fk{gl}_n(k), 
\]
where $x_i\in \fk{gl}_n(k)$.
\item Using these identifications and concrete matrix calculations, we can explicitly write 
\[
\cal{M}_{\mu}:=\cal{M}'_{G}\cap \Gr_{G,\mu}
\]
as a set of matrices for each small $\mu$.
\item From these descriptions, we can identify which nilpotent orbits correspond to $\pi^{\dagger}_G(\cal{M}_{\mu})$ (see \cite[Table 2]{AH}), and we find some geometric properties of  $\pi^{\dagger}_G|_{M_\mu}$.
\item In particular, we have $\pi_G^{\dagger}(\cal{M}_G)\subset \cal{N}_G$, and the unique maximal nilpotent orbit $C$ in $\pi^{\dagger}(\cal{M}_G)$ satisfies that $D:=(\pi_G)^{-1}(C)\subset \cal{M}_G$ is dense open and that $\pi_{G}|_{D}$ is finite.
Therefore, $\pi_G$ is finite by \cite[Lemma 2.6]{AH}, since every nilpotent orbit has even dimension.
\end{enumerate}
These calculations are valid for any $p$ if $G=SL_n$ and for $p>3$ if $G=Sp_n$ or $Spin_n$.
\item The case where $G$ is of type $E_6,E_7$ or $E_8$.
Considering a connected sub-diagram of the extended Dynkin diagram of $G$, we have a large simple subgroup $H\subset G$ of classical type.
Then, the proof proceeds as follows:
\begin{enumerate}
\item We write every small cocharacter of $G$ as a linear combination of the fundamental coweights.
\item We label each $\bar{G}$-orbit of $\cal{N}_G$ by the smallest Levi subalgebra that it meets via the Bala--Carter classification, which is available only for good $p$.
\item For each small cocharacter $\lambda$ of $G$, find the small cocharacters $\mu$ of $H$ such that $\Gr_{H,\mu}\subset \Gr_{G,\lambda}$.
These can be found in a combinatorial way.
\item For each $\mu$, find the nilpotent orbits in $\cal{N}_H$ contained in $\pi_H(\cal{M}_{H,\mu})$, by using the result for classical types.
\item For each such nilpotent orbit in $\cal{N}_H$, compute the Bala--Carter label of its $\bar{G}$-saturation in $\cal{N}_G$.
This can be done by using the procedure in the classical types for converting from a partition-type label to a Bala--Carter label, which is written in \cite[\S 6]{BC}.
\item We get the correspondence between small coweights and Bala--Carter labels so far.
We know the dimension of each Schubert cell and each nilpotent orbit.
See \cite[Table 3]{AH}.
\item Let $C$ be the unique maximal nilpotent orbit appearing in this argument.
Then by \cite[Table 3]{AH}, $D:=(\pi_{G}^{\dagger}|_{\cal{M}_G})^{-1}(C)$ is contained in the two disjoint maximal Schubert cells in $\Gr_{G}^{\sm}$ in types $E_6$, and contained in the unique maximal Schubert cell in $\Gr_{G}^{\sm}$ in types $E_7$ or $E_8$.
Choose a maximal Schubert cell $\Gr_{G,\lambda_0}$ in $\Gr_{G}^{\sm}$ and let $\mu_0$ denote the unique cocharacter of $H$ such that $\Gr_{H,\mu_0}\subset \Gr_{G,\lambda_0}$ and $C\subset G\cdot \pi_{H}(\cal{M}_{H,\mu_0})$.
By calculation, we have $\dim \Gr_{G,\lambda_0}=\dim C$.
This implies that the dimension of $\cal{M}_{\lambda_0} \cap D$ agrees with $\dim \Gr_{G,\lambda_0}=\dim C$.
The map 
\[
\pi^{\dagger}_{G}|_{\cal{M}_{\lambda_0} \cap D}\colon \cal{M}_{\lambda_0} \cap D\to C
\]
is a base change of the $G$-equivariant map $\pi^{\dagger}_{G}|_{\cal{M}_{\lambda_0}}\colon \cal{M}_{\lambda_0}\to \fk{g}$ to a $G$-orbit $C$ with $\dim \cal{M}_{\lambda_0} =\dim (\cal{M}_{\lambda_0}\cap D)=\dim C$.
Since $\cal{M}_{\lambda_0}$ is irreducible, it follows that $\cal{M}_{\lambda_0} \cap D$ is a $G$-orbit, and that $\cal{M}_{\lambda_0} \cap D\subset \cal{M}_{\lambda_0}$ is open dense.
Therefore, $D\subset \cal{M}_G$ is open dense, and in particular  $\pi_G^{\dagger}(\cal{M}_G)\subset \ol{\pi_G^{\dagger}(D)}\subset \ol{C}\subset \cal{N}_G$.

Moreover, since $\pi^{\dagger}_{G}|_{\cal{M}_{\lambda_0} \cap D}$ is a $G$-equivariant surjection between $G$-orbits of the same dimension, it is (not necessarily \'{e}tale in positive characteristic but) finite.
Therefore, $\pi_G|_D\colon D\to C$ is finite.
This implies that $\pi_{G}$ is finite by \cite[Lemma 2.6]{AH}, since every nilpotent orbit has even dimension.
\end{enumerate}
This argument is valid for good $p$.
\item The case where $G$ is of type $F_4$ or $G_2$.\\
In this case, $G$ is the set of fixed points of
an automorphism $\sigma$ of some larger simply-connected simple algebraic group $H$ of simply-laced type, where $\sigma$ comes from an automorphism of the Dynkin diagram.
Choose a $\sigma$-stable maximal torus $T_H\subset H$ such that $T=T_H^{\sigma}$.
Then $\bb{X}_{\bullet}^+(T)$ is 
\begin{enumerate}
\item For each small $\lambda\in \bb{X}_{\bullet}^+(T)$, find a $\mu\in \bb{X}_{\bullet}^+(T_H)$ such that  $\Gr_{G,\lambda}\subset \Gr_{H,\mu}$.
With a suitable choice of positive system, $\mu$ agrees with $\lambda\in \bb{X}_{\bullet}^+(T)\subset \bb{X}_{\bullet}^+(T_H)$.
By inspection, we observe that $\mu$ is small for $H$.
\item We have $\Gr_{G}^{\sm}\subset \Gr_{H}^{\sm}$, and it follows from \cite[Lemma 2.5]{AH} that $\cal{M}_G\subset \cal{M}_H$ and 
\[
\pi^{\dagger}_G(\cal{M}_G)\subset \pi^{\dagger}_H(\cal{M}_H)\cap \fk{g}\subset \cal{N}_H\cap \fk{g}=\cal{N}_G.
\]
Moreover, since $\pi_H$ is finite, so is $\pi_G$.
\end{enumerate}
This argument is valid for good $p$.
\end{enumerate}
\end{proof}
Put $\cal{M}_G:=(\cal{M}'_G)^{p^{-\infty}}$ and $\pi_G:=(\pi'_G)^{p^{-\infty}}$.
By the $G(k)$-equivariance of $j_{G}\colon \cal{M}_G\hookrightarrow \Gr^{\sm}_G$ and $\pi_G$, we obtain an exact functor $\Psi_G$ as in (\ref{eqn:defPhi}).

For the mixed characteristic case, we need to show the following proposition:
\begin{prop}
Let $G$ be a reductive group over $\cal{O}$.
Put $G^{\flat}:=\bar{G}\otimes_k k[[t]]$.
There is a constant $C=C_{\bar{G}}\in \bb{N}$ depending only on $\bar{G}$ such that if $v_F(p)\geq C$, then there is a canonical isomorphism 
\[
\Gr^\sm_{G}\cong \Gr^{\sm}_{G^{\flat}}.
\]
which is $L^{C}G(\cong L^{C}G^{\flat})$-equivariant.
\end{prop}
\begin{proof}
Since the number of coweights which is small is finite and $\Gr_G^{\sm}$ is closed in $\Gr_G$, we can write
\[
\Gr_G^{\sm}=\Gr_{G,\leq \mu_1}\cup \cdots \cup \Gr_{G,\leq \mu_r}
\]
for some $\mu_1,\ldots,\mu_r$.
Then Theorem \ref{thm:Grisomgeneral} proves the claim.
Note that the isomorphisms in \ref{thm:Grisomgeneral} agree on the overlaps $\Gr_{G,\leq \mu_i}\cap \Gr_{G,\leq \mu_j}$ by the canonicity condition \ref{item:canonicity:functorial}.
\end{proof}
Now we can construct $\cal{M}_G$ and $\pi_G$ as pullbacks of $\cal{M}_{G^{\flat}}$ and $\pi_{G^{\flat}}$ along this isomorphism. 
Now we obtain an exact functor $\Psi_G$ as in (\ref{eqn:defPhi}).
\subsection{Statement of main theorem}
We have the diagram
\[
\xymatrix{
\Perv_{L^+G}(\Gr^{\sm},\Qlb)\ar[r]^-{\scr{S}^\sm_G}\ar[d]_-{\Psi_G}&\Rep(\check{G},\Qlb)_{\sm}\ar[d]^-{\Phi_{\check{G}}}\\
\Perv_{\bar{G}}(\cal{N}^{p^{-\infty}},\Qlb)\ar[r]_-{\mathbb{S}_G}&\Rep(W_G,\Qlb).\\
}
\]

The main theorem is the following:
\begin{thm}\label{thm:main}
Assume that $p$ is good for $G$.
There exists a constant $C_{\bar{G}}$ depending only on $\bar{G}$ such that the following holds:
If $v_F(p)>C_{\bar{G}}$, there is a canonical isomorphism of functors:
\[
\Phi_{\check{G}} \circ \scr{S}^\sm_G \Iff \mathbb{S}_G \circ \Psi_G.
\]
\end{thm}
\subsection{Plan of proof}\label{scn:plan}
The plan is the same as \cite[\S 3]{AHR}.
We have to construct the isomorphism
\[
\alpha_G\colon \Phi_{\check{G}} \circ \scr{S}^\sm_G \Iff \mathbb{S}_G \circ \Psi_G, 
\]
as in the main theorem.

First, we will construct certain \textit{restriction functors} to a Levi subgroup $L$:
\begin{align*}
\fk{R}^G_L&\colon \Perv_{L^+G}(\Gr^{\sm}_G,\Qlb)\to \Perv_{L^+L}(\Gr^{\sm}_L,\Qlb)\\
\rm{R}^{\check{G}}_{\check{L}}&\colon \Rep(\check{G},\Qlb)_{\sm}\to \Rep(\check{L},\Qlb)_{\sm},\\
\cal{R}^G_L&\colon \Perv_{\bar{G}}(\cal{N}_G^{p^{-\infty}},\Qlb)\to \Perv_{\bar{L}}(\cal{N}_L^{p^{-\infty}},\Qlb),\\
\sf{R}^{W_G}_{W_L}&\colon \Rep(W_G,\Qlb)\to \Rep(W_L,\Qlb).
\end{align*}
The functor $\sf{R}^{W_G}_{W_L}$ is the obvious restriction functor.
The other functors will be defined later.

Next, we will define natural isomorphisms called \textit{transitivity isomorphisms}
\begin{align*}
\fk{R}^G_T&\Iff \fk{R}^L_T\circ \fk{R}^G_L,\\
\rm{R}^{\check{G}}_{\check{T}}&\Iff\rm{R}^{\check{L}}_{\check{T}}\circ \rm{R}^{\check{G}}_{\check{L}},\\
\cal{R}^G_T&\Iff \cal{R}^L_T\circ \cal{R}^G_L,\\
\sf{R}^{W_G}_{W_T}&\Iff \sf{R}^{W_L}_{W_T}\circ \sf{R}^{W_G}_{W_L}
\end{align*}
and \textit{intertwining isomorphisms}
\begin{align*}
\sf{R}^{W_G}_{W_L}\circ \Phi_{\check{G}}&\Iff \Phi_{\check{L}}\circ \rm{R}^{\check{G}}_{\check{L}},\\
\cal{R}^G_L\circ \Psi_{G}&\Iff \Psi_{L}\circ \fk{R}^G_L,\\
\rm{R}^{\check{G}}_{\check{L}}\circ \scr{S}^\sm_G&\Iff \scr{S}^\sm_L\circ \fk{R}^G_L,\\
\sf{R}^{W_G}_{W_L}\circ \bb{S}_G &\Iff \bb{S}_L\circ \cal{R}^G_L
\end{align*}
such that the following prisms are commutative:
\begin{align*}
&\vcenter{
\prism{\Rep(\check{G},\Qlb)_{\sm}}{\Phi_{\check{G}}}{\Rep(W_G,\Qlb)}
{\rm{R}^{\check{G}}_{\check{L}}}{(Intw)}{\sf{R}^{W_G}_{W_L}}
{(Tr)}{\Rep(\check{L},\Qlb)_{\sm}}{\Phi_{\check{L}}}{\Rep(W_L,\Qlb)}
{\rm{R}^{\check{L}}_{\check{T}}}{(Intw)}{\sf{R}^{W_L}_{W_T}}
{\Rep(\check{T},\Qlb)_{\sm}}{\Phi_{\check{T}}}{\Rep(W_T,\Qlb),}
{\rm{R}^{\check{G}}_{\check{T}}}{(Intw)}{\sf{R}^{W_G}_{W_T}}{(Tr)}
}
\\
&\vcenter{
\prism{\Perv_{L^+G}(\Gr^{\sm}_G,\Qlb)}{\Psi_{G}}{\Perv_{\bar{G}}(\cal{N}_G^{p^{-\infty}},\Qlb)}
{\fk{R}^G_L}{(Intw)}{\cal{R}^G_L}
{(Tr)}{\Perv_{L^+L}(\Gr^{\sm}_L,\Qlb)}{\Psi_{L}}{\Perv_{\bar{L}}(\cal{N}_L^{p^{-\infty}},\Qlb)}
{\fk{R}^L_T}{(Intw)}{\cal{R}^L_T}
{\Perv_{L^+T}(\Gr^{\sm}_T,\Qlb)}{\Psi_{T}}{\Perv_{\bar{T}}(\cal{N}_T^{p^{-\infty}},\Qlb),}
{\fk{R}^G_T}{(Intw)}{\cal{R}^G_T}{(Tr)}
}
\\
&\vcenter{
\prism{\Perv_{L^+G}(\Gr^{\sm}_G,\Qlb)}{\scr{S}^{\sm}_{G}}{\Rep(\check{G},\Qlb)_{\sm}}
{\fk{R}^G_L}{(Intw)}{\rm{R}^{\check{G}}_{\check{L}}}
{(Tr)}{\Perv_{L^+L}(\Gr^{\sm}_L,\Qlb)}{\scr{S}^{\sm}_{L}}{\Rep(\check{L},\Qlb)_{\sm}}
{\fk{R}^L_T}{(Intw)}{\rm{R}^{\check{L}}_{\check{T}}}
{\Perv_{L^+T}(\Gr^{\sm}_T,\Qlb)}{\scr{S}^{\sm}_{T}}{\Rep(\check{T},\Qlb)_{\sm},}
{\fk{R}^G_T}{(Intw)}{\rm{R}^{\check{G}}_{\check{T}}}{(Tr)}
}
\\
&\vcenter{
\prism{\Perv_{\bar{G}}(\cal{N}_G^{p^{-\infty}},\Qlb)}{\bb{S}_G}{\Rep(W_G,\Qlb)}
{\cal{R}^G_L}{(Intw)}{\sf{R}^{W_G}_{W_L}}
{(Tr)}{\Perv_{\bar{L}}(\cal{N}_L^{p^{-\infty}},\Qlb)}{\bb{S}_L}{\Rep(W_L,\Qlb)}
{\cal{R}^G_L}{(Intw)}{\sf{R}^{W_L}_{W_T}}
{\Perv_{\bar{T}}(\cal{N}_T^{p^{-\infty}},\Qlb)}{\bb{S}_T}{\Rep(W_T,\Qlb)}
{\cal{R}^G_L}{(Intw)}{\sf{R}^{W_G}_{W_T}}{(Tr)}
}
\end{align*}
where (Tr) and (Intw) mean the transitivity isomorphisms and the intertwining isomorphisms, respectively.
For the intertwining isomorphisms and commutativities of prisms, the same argument as in \cite[\S 5,6,7]{AHR} can be used. 

And then, we will construct the isomorphism in the case where $G=T$ or $G$ is semisimple of rank 1.
If $G$ is general, then we have the isomorphism
\begin{align*}
\phi_{G,T}\colon \sf{R}^{W_G}_{W_T}\circ \Phi_{\check{G}} \circ \scr{S}^\sm_G&\overset{\text{(Intw)}}{\Iff}\Phi_{\check{T}} \circ \rm{R}^{\check{G}}_{\check{T}}\circ \scr{S}^\sm_G\\
&\overset{\text{(Intw)}}{\Iff}\Phi_{\check{T}} \circ \scr{S}^\sm_T \circ \fk{R}^G_T\\
&\overset{\alpha_T}{\Iff} \mathbb{S}_T \circ \Psi_T\circ \fk{R}^G_T\\
&\overset{\text{(Intw)}}{\Iff}\mathbb{S}_T \circ \cal{R}^G_T \circ\Psi_G\\
&\overset{\text{(Intw)}}{\Iff}\sf{R}^{W_G}_{W_T}\circ \mathbb{S}_G \circ\Psi_G
\end{align*}
for a maximal torus $T$,
and similarly, 
\begin{align*}
\phi_{G,L}\colon \sf{R}^{W_G}_{W_L}\circ \Phi_{\check{G}} \circ \scr{S}^\sm_G\Iff \sf{R}^{W_G}_{W_L}\circ \mathbb{S}_G \circ\Psi_G
\end{align*}
for any Levi subgroup $L$ that is semisimple of rank 1.
The commutativity of the prisms implies the equality
\[
\sf{R}^{W_L}_{W_T}\phi_{G,L}=\phi_{G,T}.
\]

This means that for any $M\in \Perv_{L^+G}(\Gr^{\sm}_G,\Qlb)$, there is an isomorphism 
\begin{align}\label{eqn:mainisomfor}
\sf{R}^{W_G}_{W_T}\circ \Phi_{\check{G}} \circ \scr{S}^\sm_G(M)\overset{\phi_{G,T}}{\cong} \sf{R}^{W_G}_{W_T}\circ \mathbb{S}_G \circ\Psi_G(M)
\end{align}
which is $W_L$-equivariant for any Levi subgroup $L$ which is semisimple of rank 1.
It follows that (\ref{eqn:mainisomfor}) is $W_G$-equivariant since $W_L$'s generate $W_G$.
\subsection{Definition of restriction functors and transitivity isomorphisms}\label{scn:ResTr}
 For a Levi subgroup $L$ in $G$, we want to define the restriction functors
\begin{align*}
\fk{R}^G_L&\colon \Perv_{L^+G}(\Gr_G^\sm,\Qlb)\to \Perv_{L^+L}(\Gr_L^\sm,\Qlb),\\
\rm{R}^{\check{G}}_{\check{L}}&\colon \Rep(\check{G},\Qlb)_{\mathrm{sm}}\to \Rep(\check{L},\Qlb)_{\mathrm{sm}},\\
\cal{R}^G_L&\colon \Perv_{\bar{G}}(\cal{N}_{G}^{p^{-\infty}},\Qlb)\to \Perv_{\bar{L}}(\cal{N}_{L}^{p^{-\infty}},\Qlb),\\
\sf{R}^{W_G}_{W_L}&\colon \Rep(W_G,\Qlb)\to \Rep(W_L,\Qlb)
\end{align*}
and the natural isomorphisms called transitivity isomorphisms
\begin{align*}
\fk{R}^G_T &\Iff \fk{R}^L_T\circ \fk{R}^G_L,& \rm{R}^{\check{G}}_{\check{T}} &\Iff \rm{R}^{\check{L}}_{\check{T}}\circ \rm{R}^{\check{G}}_{\check{L}},\\
\cal{R}^G_T &\Iff \cal{R}^L_T\circ \cal{R}^G_L,&\sf{R}^{W_G}_{W_T} &\Iff \sf{R}^{W_L}_{W_T}\circ \sf{R}^{W_G}_{W_L}.
\end{align*}
\subsubsection{Category $\Rep(W_G,\Qlb)$}
 As $W_L$ can be regarded as a subgroup of $W_G$, we can define a restriction functor as the usual restriction of group action.
Furthermore, the transitivity isomorphism is simply the identity morphism.
\subsubsection{Category $\Perv_{L^+G}(\Gr^\sm_G,\Qlb)$}\label{ssc:ResSat}
Before we define the restriction functor on $\Perv_{L^+G}(\Gr^\sm_G,\Qlb)$, we define the restriction functor
\[
\overline{\fk{R}}^G_L\colon \Perv_{L^+G}(\Gr_G,\Qlb)\to \Perv_{L^+L}(\Gr_L,\Qlb).
\]
\subsubsection*{Restriction functor on $\Perv_{L^+G}(\Gr_G,\Qlb)$}\label{ssc:defofRGLbar}
There is a diagram of algebraic groups 
\begin{align}\label{LPG}
\xymatrix{
L&P \ar@{->>}[l]\ar@{^{(}->}[r]&G
}
\end{align}
where the first morphism is the natural projection, and the second is the inclusion.
This induces the diagram of affine Grassmannians
\[
\xymatrix{
\Gr_L &\Gr_P \ar[l]_-{q_P}\ar[r]^-{i_P}&\Gr_G
}
\]
where $q_P$ induces a bijection between the set of connected components.
First, define a functor $\widetilde{\fk{R}}^G_L\colon D^b(\Gr_G,\Qlb)\to D^b(\Gr_L,\Qlb)$ as the composition
\[
\xymatrix{
D^b(\Gr_G,\Qlb)\ar[r]^{(i_P)^!}&D^b(\Gr_P,\Qlb)\ar[r]^{(q_P)_*}& D^b(\Gr_L,\Qlb).
}
\]
We also define its equivariant version $\dwt{\fk{R}}^G_L\colon D^b_{L^+G}(\Gr_G,\Qlb)\to D^b_{L^+L}(\Gr_L,\Qlb)$ as the composition
\[
\xymatrix{
D^b_{L^+G}(\Gr_G)\ar[r]^{\sf{For}^{L^+G}_{L^+P}}
&D^b_{L^+P}(\Gr_G)\ar[r]^{(i_P)^!}
&D^b_{L^+P}(\Gr_P)\ar[r]^{(q_P)_*}
&D^b_{L^+P}(\Gr_L)\ar[r]^{\sf{For}^{L^+P}_{L^+L}}
&D^b_{L^+L}(\Gr_L).
}
\]
There is an isomorphism 
\begin{align}\label{eqn:ForSatResCompatible}
\wt{\fk{R}}^G_L\circ \sf{For}^{L^+G} \Iff \sf{For}^{L^+L}\circ \dwt{\fk{R}}^G_L
\end{align}
defined by 
\[
\xymatrix@R=30pt{
D^b_{L^+G}(\Gr_G)\ar[r]^{\sf{For}^{L^+G}_{L^+P}}\ar[rd]_{\sf{For}}&D^b_{L^+P}(\Gr_G)\ar[r]^{(i_P)^!}\ar[d]^{\sf{For}}\ar@{}[dl]|<<<<{\wong{(Tr)}}&D^b_{L^+P}(\Gr_P)\ar[r]^{(q_P)_*}\ar[d]^{\sf{For}}\ar@{}[dl]|{\wong{(For)}}
&D^b_{L^+P}(\Gr_L)\ar[r]^{\sf{For}^{L^+P}_{L^+L}}\ar[d]^{\sf{For}}\ar@{}[dl]|{\wong{(For)}}\ar@{}[rd]|<<<<{\wong{(Tr)}}&D^b_{L^+L}(\Gr_L)\ar[ld]^{\sf{For}}
\\
&D^b(\Gr_G)\ar[r]^{(i_P)^!}&D^b(\Gr_P)\ar[r]^{(q_P)_*}&D^b(\Gr_L)&
}
\]
We want to consider the transitivity for the functors $\widetilde{\fk{R}}^G_L$ and $\dwt{\fk{R}}^G_L$.
For that, notice the following cartesian squares:
\begin{lemm}\label{lem:Grcartesian}
The commutative squares
\begin{align}
\vcenter{
\xymatrix{
\Gr_B\ar[r]\ar[d]&\Gr_P\ar[d]\\
\Gr_{B_L}\ar[r]&\Gr_L,
}
}\label{square:Grcartesian}\\
\vcenter{
\xymatrix{
\Gr_B^\circ\ar[r]\ar[d]&\Gr_P^\circ\ar[d]\\
\Gr_{B_L}^\circ\ar[r]&\Gr_L^\circ
}
}\label{Grcirccartesian}
\end{align}
are cartesian.
\end{lemm}
\begin{proof}
It suffices to show that
\[
\xymatrix{
\Gr_B^\circ \ar[r]^a\ar[d]&\Gr_P\ar[d]^{q_P}\\
\Gr_{B_L}^\circ \ar[r]_b&\Gr_L
}
\]
is cartesian because of the equivariance of the morphisms.
Recall that the morphisms $a,b$ are locally closed embeddings and that $\Gr_B^{\circ}=\Gr_U$ and $\Gr_{B_L}^{\circ}=\Gr_{U_{B_L}}$ hold.
Hence it remains to prove
\[
a(\Gr_U)= \{x \in \Gr_P\mid q_P(x)\in b(\Gr_{U_{B_L}})\},
\]
but it can be easily checked.
\end{proof}
Thus we obtain the following diagram:
\begin{align}\label{diagram:GrGPLBBLT}
\vcenter{
\xymatrix{
\Gr_G\ar@{<-}[r]\ar@{<-}[rd]&\Gr_P\ar[r]\ar@{<-}[d]\ar@{}[rd]|{\square}&\Gr_L\ar@{<-}[d]\\
&\Gr_B\ar[r]\ar[rd]&\Gr_{B_L}\ar[d]\\
&&\Gr_T.
}
}
\end{align}
\begin{cor} \label{Sattransderivedtilde}
There is a natural isomorphism
\begin{align}
\widetilde{\fk{R}}^G_T\Iff \widetilde{\fk{R}}^L_T\circ \widetilde{\fk{R}}^G_L&\colon D^b(\Gr_G,\Qlb)\to D^b(\Gr_T,\Qlb),\label{eqn:Sattransderivedtilde}\\
\dwt{\fk{R}}^G_T\Iff \dwt{\fk{R}}^L_T\circ \dwt{\fk{R}}^G_L&\colon D^b_{L^+G}(\Gr_G,\Qlb)\to D^b_{L^+T}(\Gr_T,\Qlb).\label{eqn:Sattransderivedtildeequiv}
\end{align}
\end{cor}
\begin{proof}
We can define these isomorphisms by the following pasting diagrams:
\[
\xymatrix@R=30pt{
D^b(\Gr_G)\ar[r]^{(\cdot)^!}\ar[rd]_{(\cdot)^!}&D^b(\Gr_P)\ar[r]^{(\cdot)_*}\ar[d]^{(\cdot)^!}\ar@{}[dl]|<<<<{\wong{(Co)}}&D^b(\Gr_L)\ar[d]^{(\cdot)^!}\ar@{}[dl]|{\wong{(BC)}}\\
&D^b(\Gr_B)\ar[r]^{(\cdot)_*}\ar[rd]_{(\cdot)_*}&D^b(\Gr_{B_L})\ar[d]^{(\cdot)_*}\ar@{}[dl]|<<<<{\wong{(Co)}}\\
&&D^b(\Gr_T),
}
\]
{\footnotesize
\[
\xymatrix@R=30pt{
D^b_{L^+G}(\Gr_G)\ar[r]^{\sf{For}^{L^+G}_{L^+P}}\ar[rd]_{\sf{For}^{L^+G}_{L^+B}}
&\ar@{}[dl]|<<<<{\wong{(Tr)}}\ar@{}[rd]|{\wong{(For)}}D^b_{L^+P}(\Gr_G)\ar[r]^{(\cdot)^!}\ar[d]^{\sf{For}^{L^+P}_{L^+B}}
&D^b_{L^+P}(\Gr_P)\ar[r]^{(\cdot)_*}\ar[d]^{\sf{For}^{L^+P}_{L^+B}}\ar@{}[rd]|{\wong{(For)}}
&D^b_{L^+P}(\Gr_L)\ar[r]^{\sf{For}^{L^+P}_{L^+L}}\ar[d]^{\sf{For}^{L^+P}_{L^+B}}\ar@{}[rd]|{\wong{(Tr)}}
&D^b_{L^+L}(\Gr_L)\ar[d]^{\sf{For}^{L^+L}_{L^+B_L}}
\\
&D^b_{L^+B}(\Gr_G)\ar[r]^{(\cdot)^!}\ar[rd]_{(\cdot)^!}
&\ar@{}[rd]|{\wong{(BC)}}\ar@{}[dl]|<<<<{\wong{(Co)}}\ar@{}[rd]|{\wong{(BC)}}D^b_{L^+B}(\Gr_P)\ar[r]^{(\cdot)_*}\ar[d]^{(\cdot)^!}
&D^b_{L^+B}(\Gr_L)\ar[r]^{\sf{For}^{L^+B}_{L^+B_L}}\ar[d]^{(\cdot)^!}\ar@{}[rd]|{\wong{(For)}}
&D^b_{L^+B_L}(\Gr_L)\ar[d]^{(\cdot)^!}
\\
&&D^b_{L^+B}(\Gr_B)\ar[r]^{(\cdot)_*}\ar[rd]_{(\cdot)_*}
&\ar@{}[dl]|<<<<{\wong{(Co)}}\ar@{}[rd]|{\wong{(For)}}D^b_{L^+B}(\Gr_{B_L})\ar[r]^{\sf{For}^{L^+B}_{L^+B_L}}\ar[d]^{(\cdot)_*}\ar@{}[rd]|{\wong{(For)}}
&D^b_{L^+B_L}(\Gr_{B_L})\ar[d]^{(\cdot)_*}
\\
&&&D^b_{L^+B}(\Gr_T)\ar[r]^{\sf{For}^{L^+B}_{L^+B_L}}\ar[rd]_{\sf{For}^{L^+B}_{L^+T}}&\ar@{}[dl]|<<<<{\wong{(Tr)}}D^b_{L^+B_L}(\Gr_T)\ar[d]^{\sf{For}^{L^+B_L}_{L^+T}}
\\
&&&&D^b_{L^+T}(\Gr_T).
}
\]
}
\end{proof}
The connected components of $\Gr_L$ correspond to the characters of $Z(\wt{L})$, where $\wt{L}\subset \check{G}$ is the Levi subgroup containing $\check{T}$ whose roots are dual to those of $L$.
Let us write $(\Gr_L)_{\chi}$ for the connected component of $\Gr_L$ corresponding to $\chi\in \mathbb{X}^\bullet(Z(\wt{L}))$.
Set $\rho_{GL}:=\rho_G-\rho_L$, where $\rho_G,\rho_L$ are the half sum of positive roots of $G,L$, respectively.
We define a functor $\overline{\fk{R}}^G_L\colon D^b(\Gr_G,\Qlb)\to D^b(\Gr_L,\Qlb)$ by
\[
\overline{\fk{R}}^G_L(M)=\bigoplus_{\chi\in \mathbb{X}^\bullet(Z(\wt{L}))}\widetilde{\fk{R}}^G_L(M)|_{(\Gr_L)_\chi}[\langle\chi,2\rho_{GL}\rangle]
\]
and its equivariant version $\wt{\overline{\fk{R}}}^G_L\colon D^b_{L^+G}(\Gr_G,\Qlb)\to D^b_{L^+L}(\Gr_L,\Qlb)$ by
\[
\wt{\overline{\fk{R}}}^G_L(M)=\bigoplus_{\chi\in \mathbb{X}^\bullet(Z(\wt{L}))}\dwt{\fk{R}}^G_L(M)|_{(\Gr_L)_\chi}[\langle\chi,2\rho_{GL}\rangle].
\]
Then we have an isomorphism
\begin{align}\label{eqn:ForSatResbarCompatible}
\ol{\fk{R}}^G_L\circ \sf{For}^{L^+G} \Iff \sf{For}^{L^+L}\circ \wt{\ol{\fk{R}}}^G_L
\end{align}
by shifting (\ref{eqn:ForSatResCompatible}).
Set
\[
\sf{F}_G:=H^*(\Gr_G,-)\colon \Perv_{L^+G}(\Gr,\Qlb) \to \mathrm{Vect}_{\Qlb}
\]
where $\mathrm{Vect}_{\Qlb}$ is the category of finite dimensional $\Qlb$-vector spaces.

Recall the following standard result, see \cite[Proposition 5.3.29]{BD} or \cite[\S 4.1]{AHR}:
\begin{lemm} \label{Satres}\hfill
\begin{enumerate}
\item[(i)] $\wt{\overline{\fk{R}}}^G_L$ is exact with respect to the perverse $t$-structure.
\item[(ii)] The functor $\overline{\fk{R}}_L^G$ induces a functor
\[
\overline{\fk{R}}^G_L\colon \Perv_{L^+G}(\Gr_G,\Qlb)\to \Perv_{L^+L}(\Gr_L,\Qlb).
\]
\item[(iii)] There is a natural isomorphism
\begin{align} \label{Sattrans}
\overline{\fk{R}}^G_T\Iff \overline{\fk{R}}^L_T\circ \overline{\fk{R}}^G_L\colon \Perv_{L^+G}(\Gr_G,\Qlb)\to \Perv_{L^+T}(\Gr_T,\Qlb).
\end{align}
\end{enumerate}
\end{lemm}
Now we get the restriction functor
\[
\overline{\fk{R}}^G_L\colon \Perv_{L^+G}(\Gr_G,\Qlb)\to \Perv_{L^+L}(\Gr_L,\Qlb)
\]
admitting a transitivity isomorphism \ref{Sattrans}.
\subsubsection*{Restriction functor on $\Perv_{L^+G}(\Gr^\sm_G,\Qlb)$}
Let us define the restriction functor on $\Perv_{L^+G}(\Gr^\sm_G,\Qlb)$ using the functor $\overline{\fk{R}}^G_L$.
Set
\[
\Gr_P^\sm:=\Gr_P^\circ \cap (i_P)^{-1}(\Gr_G^\sm).
\]
\begin{lemm}
There is an inclusion $q_P(\Gr_P^\sm)\subset \Gr_L^\sm$.
\end{lemm}
\begin{proof}
Assume the contrary.
Then there exists $\lambda\in \bb{X}^{\bullet}(T)$ such that $\lambda$ is not small for $\check{L}$ and $q_P(\Gr_P^\sm)\cap S_{L,\lambda}\neq \varnothing$, where $S_{L,\lambda}$ is the semi-infinite orbit defined in \S\ref{ssc:affgrass}.

From the cartesian square (\ref{square:Grcartesian}), it follows that $i_P(q_P^{-1}(S_{L,\lambda}))=S_{G,\lambda}$.
Hence we have an inclusion
\[
(\varnothing \neq) i_P(q_P^{-1}(q_P(\Gr_P^\sm)\cap S_{L,\lambda}))\subset \Gr^{\sm}_G\cap S_{G,\lambda}.
\]
This implies that there exists a $\mu\in \bb{X}^{\bullet}(T)$ small for $\check{G}$ such that $\Gr_{G,\mu}\cap S_{G,\lambda}\neq \varnothing$.

By \cite[Corollary 2.8]{Zhumixed} or \cite[Theorem 3.2]{MV}, it means that $\lambda$ is in the convex hull of $W_G\cdot \mu$.
Therefore the convex hull of $W_L\cdot \lambda$ is contained in the convex hull of $W_G\cdot \mu$, which contradicts the fact that $\lambda$ is not small for $\check{L}$.
\end{proof}
Then we have the following commutative diagram
\begin{align}\label{Grdiag}
\vcenter{
\xymatrix{
\Gr^\sm_L\ar[d]_{f_L^\circ}&\Gr_P^\sm \ar[d]_{f_P^\circ}\ar[l]_{q_P^\sm}\ar[r]^{i_P^\sm}\ar@{}[rd]|{\square}&\Gr_G^\sm \ar[d]^{f_G^\circ}\\
\Gr^\circ_L\ar[d]_{z_L}\ar@{}[rd]|{\square}&\Gr_P^\circ\ar[d]_{z_P}\ar[l]_{q_P^\circ}\ar[r]_{i_P^\circ}&\Gr_G^\circ\ar[d]^{z_G}\\
\Gr_L&\Gr_P\ar[l]^{q_P}\ar[r]_{i_P}&\Gr_G
}
}
\end{align}
where all vertical arrows are inclusions.
The top right square is cartesian by the definition of $\Gr_P^\sm$, and the bottom left square is cartesian since $q_P$ induces a bijection between the set of connected components.
Put
\begin{align*}
\underline{\fk{R}}^G_L&:=(q_P^\circ)_*\circ (i_P^\circ)^!\colon D^b(\Gr_G^\circ,\Qlb)\to D^b(\Gr_L^\circ,\Qlb),\\
\fk{R}^G_L&:=(q_P^\sm)_*\circ (i_P^\sm)^!\colon D^b(\Gr_G^\sm,\Qlb)\to D^b(\Gr_L^\sm,\Qlb).
\end{align*}
Then we have the following lemma:
\begin{lemm}\hfill
\begin{enumerate}
\item[(i)] There is a natural isomorphism
\begin{align}\label{eqn:zGIntw}
(z_L)^!\circ \overline{\fk{R}}^G_L\Iff \underline{\fk{R}}^G_L\circ (z_G)^!.
\end{align}
In particular, $\underline{\fk{R}}^G_L$ induces a functor 
\[
\underline{\fk{R}}^G_L\colon \Perv_{L^+G}(\Gr_G^\circ)\to \Perv_{L^+L}(\Gr_L^\circ).
\]
\item[(ii)] There is a natural isomorphism
\begin{align}\label{eqn:fGIntw}
(f_L^\circ)_*\circ \fk{R}^G_L\Iff \underline{\fk{R}}^G_L\circ (f_G^\circ)_*.
\end{align}
In particular, combining with (i), $\fk{R}^G_L$ induces a functor 
\[
\fk{R}^G_L\colon \Perv_{L^+G}(\Gr_G^\sm)\to \Perv_{L^+L}(\Gr_L^\sm).
\]
\end{enumerate}
\end{lemm}
\begin{proof}
We can show this by the same argument as \cite[Lemma 4.3, Lemma 4.5]{AHR}
\end{proof}
In order to construct the transitivity isomorphism for $\fk{R}^G_L$, we want to find cartesian diagrams:
\begin{lemm}\label{Grsmcartesian}
The commutative square
\[
\xymatrix{
\Gr_B^\sm\ar[r]\ar[d]&\Gr_P^\sm\ar[d]\\
\Gr_{B_L}^\sm\ar[r]&\Gr_L^\sm
}
\]
is cartesian.
\end{lemm}
\begin{proof}
We want to prove that
\[
\xymatrix{
\Gr_B^\sm\ar[r]\ar[d]&\Gr_P^\sm\ar[d]\\
\Gr_{B_L}^\circ \ar[r]&\Gr_L^\circ \\
}
\]
is cartesian.
For that, it suffices to show that the following two squares are cartesian: 
\[
\vcenter{
\xymatrix{
\Gr_B^\sm\ar[r]\ar[d]&\Gr_P^\sm\ar[d]\\
\Gr_B^\circ \ar[r]&\Gr_P^\circ \\
}
}
\text{ and }
\vcenter{ 
\xymatrix{
\Gr_B^\circ \ar[r]\ar[d]&\Gr_P^\circ\ar[d]\\
\Gr_{B_L}^\circ \ar[r]&\Gr_L^\circ. \\
}
}
\]
The left square is cartesian by the definition of $\Gr_B^\sm,\Gr_P^\sm$, and the right square is cartesian by Lemma \ref{lem:Grcartesian}.
\end{proof}

By the same argument as Corollary \ref{Sattransderivedtilde}, we obtain the transitivity isomorphism for $\fk{R}^G_L$ from Lemma \ref{Grsmcartesian}:
\begin{cor}
There is a natural isomorphism
\begin{align*}
\fk{R}^G_T\Iff \fk{R}^L_T\circ \fk{R}^G_L\colon \Perv_{L^+G}(\Gr_G^{\sm},\Qlb)\to \Perv_{L^+T}(\Gr_T^{\sm},\Qlb).
\end{align*}
\end{cor}
\subsubsection{Category $\Rep(\check{G},\Qlb)_{\sm}$}\label{ssc:RepGcheckrestrictions}
We want to construct a restriction functor for $\Rep(\check{G},\Qlb)_{\sm}$ using the geometric Satake correspondence.

By \cite[Remark I.2.14, Proposition VI.9.6]{FS}, there is a monoidal structure on $\overline{\fk{R}}^G_L$ (see also \cite[Theorem 6.5]{B}).
Thus we have a morphism of algebraic groups
\begin{align*}
\iota^{\check{G}}_{\check{L}}\colon \check{L}&=\mathrm{Aut}^{\otimes}(\sf{F}_L)\\
&\to \mathrm{Aut}^{\otimes}(\sf{F}_L \circ \overline{\fk{R}}^G_L)\\
&\cong \mathrm{Aut}^{\otimes}(\sf{F}_G)\\
&=\check{G}.
\end{align*}
Since the following diagram of natural isomorphisms
\[
\xymatrix{
\sf{F}_G\ar@{<=>}[r]\ar@{<=>}[d]& \sf{F}_T\circ \overline{\fk{R}}^G_T\ar@{<=>}[d]\\
\sf{F}_L\circ \overline{\fk{R}}^G_L\ar@{<=>}[r]&\sf{F}_T\circ \overline{\fk{R}}^L_T\circ \overline{\fk{R}}^G_L
}
\]
is commutative, we have 
\begin{align}\label{eqn:iotatrans}
\iota^{\check{G}}_{\check{T}}=\iota^{\check{G}}_{\check{L}}\circ \iota^{\check{L}}_{\check{T}}.
\end{align}
Since $\iota^{\check{G}}_{\check{T}}$ and $\iota^{\check{L}}_{\check{T}}$ is a closed immersion of a maximal torus by \cite[\S 2.5]{Zhumixed}, it follows that $\iota^{\check{G}}_{\check{L}}$ is a closed immersion.
Thus we can identify $\wt{L}$ with $\check{L}$.

Define
\[
\overline{\rm{R}}^{\check{G}}_{\check{L}}:=(\iota^{\check{G}}_{\check{L}})_*\colon \Rep(\check{G},\Qlb)\to \Rep(\check{L},\Qlb).
\]
The equality (\ref{eqn:iotatrans}) implies that
\[
\overline{\rm{R}}^{\check{G}}_{\check{T}}=\overline{\rm{R}}^{\check{L}}_{\check{T}}\circ \overline{\rm{R}}^{\check{G}}_{\check{L}}.
\]
Put
\[
\underline{\rm{R}}^{\check{G}}_{\check{L}}:=(-)^{Z(\check{L})}\circ \ol{\rm{R}}^{\check{G}}_{\check{L}}\circ \mathbf{I}_{\check{G}}\colon \Rep(\check{G})^{Z(\check{G})}\to \Rep(\check{L})^{Z(\check{L})}.
\]
From the fact that $Z(\check{G})\subset Z(\check{L})$, it follows that 
\[
\underline{\rm{R}}^{\check{G}}_{\check{L}}\circ (-)^{Z(\check{G})}=(-)^{Z(\check{L})}\circ \ol{\rm{R}}^{\check{G}}_{\check{L}}.
\]
and that
\begin{align}\label{eqn:RepGcheckZGinvtrans}
\underline{\rm{R}}^{\check{G}}_{\check{T}}=\underline{\rm{R}}^{\check{L}}_{\check{T}}\circ \underline{\rm{R}}^{\check{G}}_{\check{L}}\colon \Rep(\check{G})^{Z(\check{G})}\to \Rep(\check{T})^{Z(\check{T})}.
\end{align}
\begin{lemm}
If $V\in \Rep(\check{G})_{\sm}$, then $V':=(\ol{\rm{R}}^{\check{G}}_{\check{L}}V)^{Z(\check{L})}\in \Rep(\check{L})_{\sm}$.
\end{lemm}
\begin{proof}
Since $Z(\check{L})$ acts trivially on $V'$, all $\check{T}$-weights of $V'$ are roots of $\check{L}$.
The convex hull of weights of $V'$ is included in the convex hull of weights of $V$, and hence no weight is of the form $2\check{\alpha}$ for a root $\check{\alpha}$ of $\check{L}$.
\end{proof}
This lemma implies that there is a unique functor $\rm{R}^{\check{G}}_{\check{L}}\colon \Rep(\check{G})_{\sm}\to \Rep(\check{L})_{\sm}$. such that
\[
\underline{\rm{R}}^{\check{G}}_{\check{L}}\circ \bb{I}^0_{\check{G}}= \bb{I}^0_{\check{L}}\circ \rm{R}^{\check{G}}_{\check{L}}.
\]
Then by (\ref{eqn:RepGcheckZGinvtrans}), we have 
\[
\rm{R}^{\check{G}}_{\check{T}}=\rm{R}^{\check{L}}_{\check{T}}\circ \rm{R}^{\check{G}}_{\check{L}}.
\]
\subsubsection{Category $\Perv_{\bar{G}}(\cal{N}_{G}^{p^{-\infty}})$}
By the same argument as \cite[\S 4.4]{AHR}, we get the restriction functor
\[
\cal{R}^G_L\colon \Perv_{\bar{G}}(\cal{N}_G^{p^{-\infty}})\to \Perv_{\bar{L}}(\cal{N}_L^{p^{-\infty}})
\]
(note that $\Perv_{\bar{G}}(\cal{N}_G)$ is canonically equivalent to $\Perv_{\bar{G}}(\cal{N}_G^{p^{-\infty}})$)
and the restriction isomorphism
\[
\cal{R}^G_T\Iff \cal{R}^L_T\circ \cal{R}^G_L\colon \Perv_{\bar{G}}(\cal{N}_G^{p^{-\infty}})\to \Perv_{\bar{T}}(\cal{N}_T^{p^{-\infty}}).
\]
\subsection{Intertwining isomorphisms and commutativity of prism}\label{scn:Intw}
By the same argument as \cite[\S 5, 6, 7]{AHR}, we can construct the intertwining isomorphisms satisfying the commutativity of the prisms:
\begin{thm}
There are natural isomorphisms
\begin{align*}
\sf{R}^{W_G}_{W_L}\circ \Phi_{\check{G}}&\Iff \Phi_{\check{L}}\circ \rm{R}^{\check{G}}_{\check{L}},\\
\cal{R}^G_L\circ \Psi_{G}&\Iff \Psi_{L}\circ \fk{R}^G_L,\\
\rm{R}^{\check{G}}_{\check{L}}\circ \scr{S}^\sm_G&\Iff \scr{S}^\sm_L\circ \fk{R}^G_L,\\
\sf{R}^{W_G}_{W_L}\circ \bb{S}_G &\Iff \bb{S}_L\circ \cal{R}^G_L
\end{align*}
which make the four prisms in \S \ref{scn:plan} commute.
\end{thm}
\subsection{The case $G$ is a torus or semisimple of rank 1}\label{scn:semisimplerank1}
As explained in \S \ref{scn:plan}, all that remains to prove the main theorem is the case $G=T$ or $G$ is semisimple of rank 1.
Assume that $p$ is good for $G$ and that $v_F(p)$ is sufficiently large so that $\Psi_G$ can be defined.
\subsubsection{The case $G=T$}
If $G=T$, then $\Gr^{\sm}_T$ and $\cal{N}_T^{p^{-\infty}}$ are both single points, and $\Psi_T$ is the canonical identification.
Moreover, $\Phi_{\check{T}}\circ \scr{S}_T^{\sm}$ is the equivalence $H^0\colon \Perv_{L^+T}(\Gr^{\sm}_T,\Qlb)\to \rm{Vect}_{\Qlb}$ and $\bb{S}_T$ is canonically isomorphic to $H^0\colon \Perv_{\bar{T}}(\cal{N}_T^{p^{-\infty}},\Qlb)\to \rm{Vect}_{\Qlb}$.
Hence we have a canonical isomorphism
\[
\alpha_T\colon \Phi_{\check{T}}\circ \scr{S}^{\sm}_T\Iff \bb{S}_T\circ \Psi_{T}.
\]
\subsubsection{The case $G$ is semisimple of rank 1}
We may assume $G=\rm{PGL}_2$ since all the functors are invariant under the replacement of $G$ by $G/Z(G)$.
Using $\alpha_T$, there is an isomorphism
\[
\phi_{G,T}\colon \sf{R}^{W_G}_{W_T}\circ \Phi_{\check{G}}\circ \scr{S}^{\sm}_G\Iff \sf{R}^{W_G}_{W_T}\circ \bb{G}_T\circ \Psi_{G}
\]
as in \S \ref{scn:plan}.
It suffices to show that this isomorphism is $W_G$-equivariant.

Identify $\bb{X}_{\bullet}(T)$ with $\bb{Z}$.
Then 
\[
\Gr^{\sm}=\Gr_0\amalg \Gr_2.
\]
Let $j_i\colon \Gr_{i}\hookrightarrow \Gr\ (i=0,2)$ be the inclusion map.
Set
\[
\rm{IC}_i:=(j_i)_{!*}(\Qlb[i]).
\]
Recall that $\Perv_{L^+G}(\Gr^{\sm}_G,\Qlb)$ is semisimple and simple objects are $\rm{IC}_0$ and $\rm{IC}_2$.
Therefore, it suffices to show that $\phi_{G,T}$ is $W_G$-equivariant for the objects $\rm{IC}_0,\rm{IC}_2$. 

Since $G=\rm{PGL}_2$, the map $\pi\colon \cal{M}_G\to \cal{N}_G^{p^{-\infty}}$ is an isomorphism (see \cite{AH}).
Hence $\Phi_{G}(\rm{IC}_i)$ is canonically isomorphic to the intermediate extension in $\cal{N}^{p^{-\infty}}$ of ${\Qlb}_{\cal{N}^{p^{-\infty}}\cap \Gr_i}$.

Note that $\cal{N}^{p^{-\infty}}\cap \Gr_0$ and $\cal{N}^{p^{-\infty}}\cap \Gr_2$ are the $\bar{G}$-orbits in $\cal{N}_G^{p^{-\infty}}$.
From the theory of the Springer correspondence, it follows that $\bb{S}_G(\Phi_{G}(\rm{IC}_0))$ is isomorphic to the trivial representation of $W=S_2$, and $\bb{S}_G(\Phi_{G}(\rm{IC}_2))$ is isomorphic to the sign representation of $W=S_2$.

On the other hand, from the theory of the geometric Satake, $\scr{S}^{\sm}_G(\rm{IC}_i)$ is the irreducible representation of $\check{G}$ with highest weight $i$.
By the Schur-Weyl duality, $\Phi_{\check{G}}(\scr{S}^{\sm}_G(\rm{IC}_0))$ is isomorphic to the trivial representation of $W=S_2$, and $\Phi_{\check{G}}(\scr{S}^{\sm}_G(\rm{IC}_2))$ is isomorphic to the sign representation of $W=S_2$.
As a result, $\bb{S}_G(\Phi_{G}(\rm{IC}_i))$ and $\Phi_{\check{G}}(\scr{S}^{\sm}_G(\rm{IC}_i))$ are irreducible as $W$-modules and their isomorphism classes coincide.

But $W=S_2$ acts on an irreducible representation by scalar multiplication, it follows that the isomorphism of vector spaces
\[
\Phi_{\check{G}}(\scr{S}^{\sm}_G(\rm{IC}_i)) \overset{\phi_{G,T}}{\cong} \bb{S}_G(\Phi_{G}(\rm{IC}_i))
\]
is automatically $W$-equivariant.
\bibliographystyle{test}
\bibliography{satakespringerbib}
\end{document}